\documentclass{article}

\usepackage{latexsym}
\usepackage{a4wide}
\usepackage{amscd}
\usepackage{graphics}
\usepackage{amsmath}
\usepackage{amssymb}
\usepackage{mathrsfs}
\usepackage{bbm}
\usepackage{amsthm}
\input xy
\xyoption{all}

\newcommand{\N}{\mathbb{N}}                     
\newcommand{\R}{\mathbb{R}}                     
\newcommand{\C}{\mathbb{C}}                     
\newcommand{\T}{\mathbb{T}}                     
\newcommand{\set}[2]{\left\{{#1}\mid{#2}\right\}}       
\newcommand{\Span}{\mathrm{span\,}}             
\newcommand{\supp}{\mathrm{supp\,}}             
\newcommand{\dom}{\mathrm{dom}\,}               

\newtheorem*{mainthm}{\sc Theorem}           
\newtheorem{thm}{\sc Theorem}[section]       
\newtheorem{cor}[thm]{\sc Corollary}        
\newtheorem{lem}[thm]{\sc Lemma}            
\newtheorem{prop}[thm]{\sc Proposition}     
\newtheorem{defn}[thm]{\sc Definition}      
\newtheorem{rem}[thm]{\sc Remark}           

\title{Global stable manifolds in holomorphic dynamics under bunching conditions}

\author{Alberto Abbondandolo and Pietro Majer}

\date{November 22, 2011}

\begin{document}

\renewcommand{\theenumi}{\roman{enumi}}
\renewcommand{\labelenumi}{(\theenumi)}

\maketitle

\begin{abstract}
We prove that the stable manifold of every point in a compact hyperbolic  invariant set of a holomorphic automorphism of a complex manifold is biholomorphic to $\C^d$, provided that a bunching condition, which is weaker than the classical bunching condition for linearizability, holds.
\end{abstract} 

\section*{Introduction}

In this paper we study a foundational question in holomorphic dynamics: are stable manifolds biholomorphic to $\C^d$?
More precisely, let $f:X\rightarrow X$ be a holomorphic automorphism of a complex manifold and let $K\subset X$ be a compact $f$-invariant set. We assume that $K$ is hyperbolic, meaning that there is a $Df$-invariant splitting of the tangent bundle of $X$ restricted to $K$
\[
TX|_K = E^s \oplus E^u
\]
such that for every $x\in K$
\[
\bigl\| Df^n(x)|_{E^s} \bigr\| \leq C\, \Lambda_s^n, \qquad 
\bigl\| Df^{-n}(x)|_{E^u} \bigr\| \leq C\, \Lambda_u^n, \qquad \forall n\in \N,
\]
for some numbers $0<\Lambda_s<1$, $0<\Lambda_u <1$, $C>0$. Here the operator norm is induced by some Riemannian metric on $M$, but the notion of hyperbolicity and the numbers $\Lambda_s$, $\Lambda_u$ do not depend on the choice of this metric (see e.g. \cite{shu87}, \cite{kh95}, or \cite{yoc95b} for general facts about hyperbolic dynamics). Then each point $x\in K$ has a stable manifold
\[
W^s(x) := \bigl\{z\in X \, \big|\, \mathrm{dist}\,(f^n(z),f^n(x)) \rightarrow 0 \mbox{ for } n\rightarrow \infty\bigr\}
\]
which is the image of an injective holomorphic immersion $W \hookrightarrow X$
of a complex manifold $W$ which is $C^{\infty}$-diffeomorphic to $\R^{2d}$, where $d$ is the complex dimension of the stable distribution $E^s$. The question is whether $W$ is also biholomorphic to $\C^d$. When this holds, with a slight abuse of terminology we shall say that $W^s(x)$ is biholomorphic to $\C^d$.

This question has a positive answer when $K$ is a hyperbolic fixed point - or more generally a periodic orbit - as shown by J.\ P.\ Rosay and W.\ Rudin in \cite{rr88}. The question for a general compact hyperbolic set $K$ was posed by E.\ Bedford in \cite{bed00}. In such a case, M.\ Jonsson and D.\ Varolin have shown in \cite{jv02} that the answer is also positive for almost every point $x\in K$, with respect to any $f$-invariant probability measure on $K$. More generally, they have proved an analogous result in the setting of partial hyperbolicity. Extensions of these positive results to sequences of holomorphic contractions
\[
f_n : B \rightarrow B, \qquad B=\set{z\in \C^d}{|z|<1}, \qquad f_n(0)=0,
\]
under weak monotonicity assumptions on the diagonal entries of $Df_n(0)$ - once $Df_n(0)$ is put into upper triangular form by applying a suitable unitary non-autonomous conjugacy - are proved in \cite{pet05}, \cite{pet07}, and \cite{aam11}. In particular, the main result of \cite{aam11} implies that the answer to Bedford's question is positive for every $x\in K$ which has well-defined Lyapunov exponents. See also \cite{wei97,for99,for04,fs04,wol05,pw05,pvw08,bdm08,aro11} for related results.

On the other hand, it has been known for a long time and it is easy to show that $W^s(x)$ is biholomorphic to $\C^d$ for every $x\in K$ when the hyperbolic set satisfies a suitable bunching condition: If the numbers $0<\Lambda_s < 1 < M_s$ are such that
\begin{equation}
\label{bunching}
\bigl\| Df^n(x)|_{E^s} \bigr\| \leq C\, \Lambda_s^n, \qquad 
\bigl\| ( Df^n(x)|_{E^s} )^{-1} \bigr\| \leq C \, M_s^n, \qquad \forall n\in \N,
\end{equation}
for some $C>0$, it is required that
\[
\Lambda_s^2 M_s < 1,
\]
or equivalently that
\[
\beta:= \frac{\log M_s}{-\log \Lambda_s} < 2.
\]
The number $\beta$ is sometimes called the bunching parameter. The hypothesis that the bunching parameter is less than 2 implies that the restriction of $f$ to the stable manifold of $K$ is locally holomorphically linearizable, and this fact immediately implies that the stable manifold of each $x\in K$ is biholomorphic to $\C^d$ (see \cite[Theorem 3]{jv02}). Similar bunching conditions are relevant in other questions of hyperbolic dynamics, such as the issue of regularity of the stable and unstable foliations, normal forms, rigidity results, ergodicity (see e.g.\ \cite[Chapter 6]{kh95}, \cite[Chapter 3]{hk02}, \cite{rru07}, \cite[Chapter 5.2.2]{kn11}, \cite{bw10}). 

It is also well-known that the condition $\beta<2$ is sharp for the linearization question: One can construct examples with $\beta=2$ where linearization is not possible (see e.g.\ \cite[Exercise 6.6.1]{kh95}). The aim of this paper is to show that the bunching value 2 is instead not critical for Bedford's question. More precisely, we shall prove the following:

\begin{mainthm}
Let $f:X\rightarrow X$ be a holomorphic automorphism of a complex manifold and let $K\subset X$ be a compact hyperbolic invariant set with stable distribution $E_s$. Assume that the condition (\ref{bunching}) holds and that
\[
\beta:= \frac{\log M_s}{-\log \Lambda_s} < 2 + \epsilon(d),
\] 
where $\epsilon(d)$ is a positive number depending on the complex dimension $d$ of the stable distribution $E^s$. Then the stable manifold $W^s(x)$ of every $x\in K$ is biholomorphic to $\C^d$.
\end{mainthm}

In particular, stable manifolds of points on 2-bunched compact hyperbolic invariant sets are biholomorphic to a complex vector space in every dimension.
The explicit expression for the number $\epsilon(d)$ is given below in Theorem \ref{principale}. The sequence $\epsilon(d)$ converges rapidly to zero for $d\rightarrow \infty$. 

We conclude this introduction by informally discussing the general strategy behind the previous results about the complex structure of $W^s(x)$ and the new ideas which are involved in the proof of the above theorem. 

Let $x\in K$. By looking at the restriction of $f$ to the local stable manifold of each point $f^n(x)$, $n\in \N$, which is biholomorphic to the unit ball $B$ of $\C^d$, one obtains a sequence of holomorphic contractive embeddings
\[
f_n : B \hookrightarrow B,
\]
which fix $0$. The stable manifold $W^s(x)$ is bihomolorphic to the ``abstract basin of attraction'' $W$ of $0$ with respect to the sequence $(f_n)$, that is the inductive limit given by the sequence of embeddings $(f_n)$ (see \cite{fs04}, or Section \ref{abasec} below). Therefore, it is enough to show that the complex manifold $W$ is biholomorphic to $\C^d$. In order to study the complex structure of such a manifold, a natural idea is to find a suitable non-autonomous conjugacy, that is a commuting diagram of the form
\[
\begin{CD}
B @>{f_n}>> B \\ @V{h_n}VV @VV{h_{n+1}}V \\ \C^d @>{g_n}>> \C^d, 
\end{CD}
\]
where $g_n$ is a sequence of automorphisms of $\C^d$ having $0$ as an attracting fixed point.
Indeed, under suitable boundedness assumptions on $(h_n)$, the above diagram induces a biholomorphisms from the abstract basin of attraction of $0$ with respect to $(f_n)$ to the one determined by $(g_n)$, which can be identified with the domain
\begin{equation}
\label{bas}
\set{z\in \C^d}{g_n \circ  g_{n-1} \circ \dots \circ g_0 (z) \rightarrow 0 \mbox{ for } n\rightarrow \infty}.
\end{equation}
Therefore, one needs to find a sequence of automorphisms $(g_n)$ for which the basin of attraction (\ref{bas}) is the whole $\C^d$ and for which the conjugacy equation
\begin{equation}
\label{eqconj}
h_{n+1} \circ f_n = g_n \circ h_n
\end{equation}
admits a suitably bounded solution $(h_n)$. 

A useful class of automorphisms which have the whole $\C^d$ as basin of attraction of $0$ is the class of ``special triangular automorphisms'', that is polynomial maps
\[
g: \C^d \rightarrow \C^d, \qquad z\mapsto z',
\]
of the form
\begin{eqnarray*}
&z_1' &= \lambda_1 z_1 + p_1(z_2,\dots,z_d) \\ 
&z_2' &= \lambda_2 z_2 + p_2(z_3,\dots,z_d) \\
&\vdots  & \qquad \vdots  \qquad \qquad \vdots \\
&z_{d-1}' &= \lambda_{d-1} z_{d-1} + p_{d-1}(z_d) \\
&z_d' &= \lambda_d z_d,
\end{eqnarray*}
where $|\lambda_j|\leq \Lambda <1$ and $p_j\in \C[z_{j+1},\dots,z_d]$ vanishes at $0$. If $(g_n)$ is a sequence of special triangular automorphisms, such that the $p_j$'s have uniformly bounded degree and coefficients and such that the upper bound $\Lambda<1$ for the absolute value of the diagonal entries $|\lambda_j|$ is also uniform, then the set (\ref{bas}) is the whole $\C^d$. 

All the proofs of the previous results (\cite{rr88}, \cite{jv02}, \cite{pet07}, \cite{aam11}) are based on finding a suitably bounded non-autonomous conjugacy with a sequence of special triangular automorphisms. The nonlinear conjugacy equation (\ref{eqconj}) becomes linear if one considers it separately on each $k$-homogeneous part of $(h_n)$, for $k=1,2,\dots$. Actually, one has to consider only finitely many of these linear equations, because of the following general fact: If the sequence $(f_n)$ satisfies
\begin{equation}
\label{ssttmm}
\|Df_n(0)\| \leq \Lambda, \qquad  \qquad \|Df_n(0)^{-1}\| \leq M,
\end{equation}
for some $0<\Lambda<1<M$ and is boundedly conjugated at the level of $k$-jets to the sequence $(g_n)$, with $k+1$ larger than the bunching parameter $\beta=-\log M/\log \Lambda$, or equivalently $\Lambda^{k+1}M<1$, then $(f_n)$ and $(g_n)$ are actually locally conjugated (see Theorem \ref{volata} below for a precise statement of this well-known fact). 

At the level of $1$-jets, the conjugacy equation (\ref{eqconj}) takes the form
\[
Dh_{n+1}(0) \circ Df_n(0) = Dg_n(0) \circ Dh_n(0).
\]
The fact that every matrix $L$ can be decomposed as $L=QR$ with $Q$ unitary and $R$ upper triangular implies that the equation
\[
H_{n+1} \circ Df_n(0) = G_n \circ H_n
\]
has a solution $(H_n)$, $(G_n)$ with $H_n$ unitary and $G_n$ upper triangular (once we fix $H_0$, the solution is unique up to multiplication by diagonal unitary matrices). This means that at the level of $1$-jets it is always possible to conjugate $(f_n)$ to a sequence of special triangular (linear) automorphisms $(g_n)$. When $\beta<2$, by the general fact stated above, this conjugacy of $1$-jets extends to a local conjugacy (that is, the sequence $(f_n)$ is linearizable), and the abstract basin of attraction of 0 with respect to $(f_n)$ is biholomorphic to $\C^d$.

When $\beta\geq 2$, one has to consider also the conjugacy equation for $k$-jets for $k\geq 2$. Under suitable weak monotonicity assumptions on the diagonal entries of the upper triangular matrix $G_n = Dg_n(0)$, one can actually find a bounded solution $(h_n)$, $(g_n)$ for such an equation, with $(g_n)$ a sequence of special triangular automorphisms of degree $k$. If $k$ is large enough, one deduces the existence of a local conjugacy (see \cite{pet07} and \cite{aam11}). In particular, the required monotonicity assumptions hold when the diagonal entries of $G_n$ have a good limiting behavior, and by Oseledec's theory one finds the aforementioned almost everywhere statement of \cite{jv02}.

However, in general if $\beta\geq 2$ the conjugacy equation at the level of $2$-jets may have no bounded solution $(h_n)$, $(g_n)$, with $(g_n)$ a sequence of special triangular automorphisms of degree 2. An example with $\beta=2$ is given in \cite[Section 4]{aam11}. Therefore, one needs to enlarge the possible normal forms for $(f_n)$ to a  larger class of sequences $(g_n)$ of automorphisms of $\C^d$ having the whole $\C^d$ as basin of attraction of $0$. 

The starting idea of this paper is to use sequences of automorphisms of the form
\begin{equation}
\label{forma}
g_n = U_{\sigma_h} \circ \tilde{g}_n \circ U_{\sigma_h}^{-1}, \qquad \mbox{for} \quad m_h \leq n < m_{h+1},
\end{equation}
where $(\tilde{g}_n)$ is a sequence of special triangular automorphisms, $(\sigma_h)\subset \mathfrak{S}_d$ is a sequence of permutations of the set $\{1,\dots,d\}$, $U_{\sigma}$ is the unitary matrix which maps $e_i$ into $e_{\sigma(i)}$ for every $i\in \{1,\dots,d\}$, and $(m_h)$ is a strictly increasing sequence of natural numbers which grows fast enough. In other words, $g_n$ is special triangular with respect to some permutation of the standard basis of $\C^d$, such a permutation being constant over larger and larger intervals of indices $n$. Indeed, as we prove in Theorem \ref{treni}, if
\begin{equation}
\label{growth}
\lim_{h\rightarrow \infty} (m_{h+1} - m_h) = +\infty  \qquad \mbox{and} \qquad \sum_{h=0}^{\infty} K^{-h} m_h = +\infty,
\end{equation}
where $K$ is the ``stable degree'' of $(\tilde{g}_n)$, then the basin of attraction of $0$ with respect to $(g_n)$ is the whole $\C^d$. In the proof of our main result we construct a conjugacy with a sequence $(g_n)$ in the above class. In order to explain why such an enlargement of the class of normal forms $(g_n)$ allows to weaken the standard bunching assumption, let us consider the particular case in which 
\[
Df_n(0) = \mathrm{Diag} \, \bigl( \lambda_n(1), \lambda_n(2), \dots, \lambda_n(d)   
\bigr) 
\]
is diagonal for every $n\in \N$. If we set
\begin{eqnarray*}
f_n(z) & = & \sum_{1\leq k \leq d} \lambda_n(k) z_k e_k + \sum_{\substack{1\leq i \leq j \leq d \\ 1\leq k \leq d}} f_n(i,j;k) z_i z_j e_k + O(|z|^3), \\
g_n(z) & = & \sum_{1\leq k \leq d} \lambda_n(k) z_k e_k + \sum_{\substack{1\leq i \leq j \leq d \\ 1\leq k \leq d}} g_n(i,j;k) z_i z_j e_k + O(|z|^3), \\
h_n(z) & = & \sum_{1\leq k \leq d}  z_k e_k + \sum_{\substack{1\leq i \leq j \leq d \\ 1\leq k \leq d}} h_n(i,j;k) z_i z_j e_k + O(|z|^3),
\end{eqnarray*}
the conjugacy equation for 1-jets is fulfilled and the one for 2-jets can be written as the system of $d^2(d+1)/2$ uncoupled equations
\begin{equation}
\label{2jets}
h_{n+1}(i,j;k) = \frac{\lambda_n(k)}{\lambda_n(i) \lambda_n(j)} h_n (i,j;k) + \frac{1} 
{\lambda_n(i) \lambda_n(j)} \bigl( g_n(i,j;k) - f_n(i,j;k) \bigr),
\end{equation}
for every $(i,j;k)$ in the set
\[
\mathbb{H} := \set{ (i,j;k)\in \N^3 }{1\leq i \leq j \leq d, \; 1\leq k \leq d},
\]
which indexes the standard basis $\{z_iz_j e_k\}$ of the space $\mathscr{H}$ of 2-homogeneous polynomial endomorphisms of $\C^d$.
By (\ref{ssttmm}), we have the bounds
\[
0<M^{-1} \leq |\lambda_n(i)| \leq \Lambda<1, \qquad \forall i\in \{1,\dots,d\}, \; \forall n\in \N.
\]
When $k$ coincides with either $i$ or $j$, the above bounds imply that
\[
\left|\frac{\lambda_n(k)}{\lambda_n(i)\lambda_n(j)} \right| \geq \frac{1}{\Lambda} > 1.
\]
Therefore, $h_{n+1}(i,j;k)$ is obtained from $h_n(i,j;k)$ by applying a uniformly expanding affine function. This fact easily implies that for such a triplet $(i,j;k)$ we may set $g_n(i,j;k)=0$ for every $n\in \N$ and find a (unique) bounded solution $(h_n(i,j;k))_{n\in \N}$ of (\ref{2jets}). 

On the other hand, if $k$ is different from $i$ and from $j$, we just have the estimate
\[
\left|\frac{\lambda_n(k)}{\lambda_n(i)\lambda_n(j)} \right| \geq \frac{1}{\Lambda^2 M},
\]
which produces a lower bound larger than 1 only if the bunching condition for linearization - $\beta<2$ - holds. It is possible to construct examples with $\beta=2$ such that the equations (\ref{2jets}) for these triplets $(i,j;k)$ do not have any bounded solution $(h_n(i,j;k))_{n\in \N}$ with $g_n(i,j;k)=0$ for every $n\in \N$ (see \cite[Section 4]{aam11}). 
This means that the subset
\[
\mathbb{V} := \set{(i,j;k)\in \mathbb{H}}{k\neq i \mbox{ and } k\neq j} 
\]
should be considered as the set of resonances for the non-autonomous conjugacy equation. The subspace of $\mathscr{H}$ consisting of those polynomial maps which can appear in a special triangular automorphism is generated by the elements of the standard basis $\{z_iz_je_k\}$ which are indexed by the subset
\[
\T := \set{(i,j;k)\in \mathbb{H}}{k< i}.
\]
Since $\T$ is a proper subset of $\mathbb{V}$, it is in general not possible to find a bounded conjugacy with a sequence of special triangular automorphism, even at the level of 2-jets (see again \cite[Section 4]{aam11}). However, the image of $\T$ by the action of the permutation group $\mathfrak{S}_d$ on $\mathbb{H}$, which is given by
\[
\sigma \cdot (i,j;k) = \bigl(\sigma^{-1}(i), \sigma^{-1}(j) ; \sigma(k) \bigr), \qquad \forall \sigma\in \mathfrak{S}_d, \; \forall (i,j;k)\in \mathbb{H}, 
\]
is the whole set of resonances $\mathbb{V}$. Such an action on $\mathbb{H}$ corresponds to the action of $\mathfrak{S}_d$ on $\mathscr{H}$ given by conjugacy by the unitary matrix $U_{\sigma}$,
\[
\sigma \cdot p = U_{\sigma} \circ p \circ U_{\sigma}^{-1}, \qquad \forall \sigma \in \mathfrak{S}_d, \; \forall p \in  \mathscr{H}.
\]
This fact suggests that it should be easier to find a bounded conjugacy to a sequence of polynomial maps $(g_n)$ of the form (\ref{forma}) and of degree 2, where $(\sigma_h)$ varies periodically among a set of permutations in $\mathfrak{S}_d$ whose action applied to $\T$ produces the whole $\mathbb{V}$. Since we also want the basin of attraction of $0$ with respect to $(g_n)$ to be the whole $\C^d$, we need that the sequence $(m_h)$ satisfies the growth condition (\ref{growth}), where $K=2^{d-1}$ is the ``stable degree'' of $(\tilde{g}_n)$.

In the general case, in which $Df_n(0)$ cannot be assumed to be diagonal, the conjugacy equation for 2-jets does not split into uncoupled scalar equations and
a delicate point is the choice of the sequence $(\sigma_h)\subset \mathfrak{S}_d$.
We now explain the role of such a choice, by making some rather imprecise and slightly incorrect statements, which are made precise and correct in the following sections. 

  If we call $u_j \in \mathscr{H}$ the homogeneous part of degree 2 of $h_{m_j}$, the conjugacy equation (\ref{eqconj}) to a sequence $(g_n)$ of the form (\ref{forma}), at the level of 2-jets, takes the form
\begin{equation}
\label{eqxu}
u_{h+1} = \mathscr{A}_h (\sigma_h \cdot u_h) - v_h, \qquad \forall h\in \N,
\end{equation}
where $(v_h)\subset \mathscr{H}$ is defined recursively and can be thought to be given, while $\mathscr{A}_h$ is the conjugacy operator by a product of $m_{h+1}-m_h$ upper triangular matrices. A natural choice for a solution $(u_h)$ of (\ref{eqxu}) is given by the expression
\begin{equation}
\label{serie}
u_h = \sum_{k=0}^{\infty} \tilde{\mathscr{A}}_h^{-1} \tilde{\mathscr{A}}_{h+1}^{-1} \dots \tilde{\mathscr{A}}_{h+k}^{-1} v_{h+k}, \qquad \mbox{where} \quad \tilde{\mathscr{A}}_h := \mathscr{A}_h \circ \sigma_h.
\end{equation}  
Indeed, this is the formula for the unique bounded solution of (\ref{eqxu}), when the $\mathscr{A}_j^{-1}$'s are uniform contractions, which is true if and only if $\beta<2$. Since we are assuming $\beta>2$, the $\mathscr{A}_j^{-1}$'s are not contractions and the convergence of (\ref{serie}) is problematic. Nevertheless, we would like to show that if $\beta-2$ is small enough, then a suitable choice of the sequence of permutations $(\sigma_h)$ makes the series (\ref{serie}) converging.  

A key ingredient here is the existence of a partial order on the set $\mathbb{H}$, which has the property that conjugacy of an element of $\mathscr{H}$ with respect to an upper triangular linear automorphism of $\C^d$ is a linear mapping on $\mathscr{H}$ which is upper triangular with respect to such an order. In particular, $\mathscr{A}_h$ is upper triangular with respect to this order. Moreover, its inverse can be decomposed as
\[
\mathscr{A}_h^{-1} = \mathscr{M}_h^0 + \mathscr{M}_h^1,
\]
where
\[
\bigl\|\mathscr{M}_h^0\bigr\| = O\bigl(\Lambda^{m_{h+1}-m_h}\bigr), \qquad
\bigl\|\mathscr{M}_h^1\bigr\| = O\bigl((\Lambda^2M)^{m_{h+1}-m_h}\bigr),
\]
and the matrix associated to $\mathscr{M}_h^1$ has support in the subset $\mathbb{W}$ of $\mathbb{H}\times \mathbb{H}$ consisting of all pairs of indices in $\mathbb{H}\times \mathbb{H}$ which are the end-points of a chain in the partial order which has non-empty intersection with the set of resonances $\mathbb{V}$. 
Since $\Lambda^2 M>1$, the terms $\mathscr{M}_h^1$'s are the parts of large norm in the decomposition of $\mathscr{A}_h^{-1}$. In order to prove that their presence does not prevent the series (\ref{serie}) from converging, the main tool is a non-trivial combinatorial nilpotency lemma, which asserts the existence of a finite sequence $\sigma_1,\sigma_2, \dots, \sigma_N$ in $\mathfrak{S}_d$ such that
\[
\mathscr{M}_0 \sigma_1 \cdot \mathscr{M}_1 \sigma_2 \cdot \mathscr{M}_2 \dots \sigma_N \cdot
\mathscr{M}_N = 0,
\]
whenever the matrices associated to the $\mathscr{M}_j$'s have support in $\mathbb{W}$. Here $N=\binom{d}{2}$.
The sequence $(\sigma_h)$ that we use is defined by periodically repeating such a finite sequence. 

By combining these statements with analytic estimates, we can prove the existence of a solution $(h_n)$, $(g_n)$ of the conjugacy equation (\ref{eqconj}) at the level of $2$-jets, where $(g_n)$ is of the form (\ref{forma}) with $m_ h=2^{(d-1)h}$ and $(h_n)$ is a sequence of polynomial maps of degree 2, tangent to some unitary map at $0$, and such that
\[
\|h_n\| \leq C\, R^n,
\]
where $R=R(d)>1$ depends on the bunching parameter $\beta=-\log M/\log \Lambda$ and tends to $1$ for $\beta\downarrow 2$. By a rescaling argument, when $R$ is close enough to $1$, more precisely when
\begin{equation}
\label{conR}
R^2 \Lambda^3 M < 1,
\end{equation}
such an estimate allows us to conclude that the abstract basin of attraction of $0$ with respect to $(f_n)$ is biholomorphic to $\C^d$. The requirement (\ref{conR}) determines the bunching condition $\beta<2+\epsilon(d)$ in our main theorem. Possible ways of improving the exponent $\epsilon(d)$ are discussed in the final Section \ref{conrem}.

\section{The structure of the conjugacy operator}

\paragraph{Pinched sequences of linear automorphisms.}
If $(L_n)_{n\in \N}$ is a sequence of linear endomorphisms of $\C^d$ and $n\geq m \geq 0$, we denote by $L_{n,m}$ the composition
\[
L_{n,m} = L_{n-1} L_{n-2} \dots L_m, \quad L_{n,n} = I,
\]
which satisfies, for any $n\geq m \geq \ell \geq 0$, 
\[
L_{n,m} L_{m,\ell} = L_{n,\ell}.
\]

\begin{defn} 
Let $\Lambda$ and $M$ be positive numbers. We say that a sequence $(L_n)$ of linear automorphisms of $\C^d$ is $(\Lambda,M)$-pinched if there exists a positive number $C$ such that 
\begin{equation}
\label{est}
\|L_{k,h}\|\leq C \,\Lambda^{k-h}, \quad \|L_{k,h}^{-1}\|\leq C\, M^{k-h}, \qquad \mbox{for any }  0 \leq h \leq k.
\end{equation}
\end{defn}

Here $\|\cdot\|$ denotes the operator norm.
By taking $k=h+1$ in (\ref{est}), we see that a pinched sequence $(L_n)$ is bounded and so is the sequence of its inverses $(L_n^{-1})$. 
Notice also that (\ref{est}) implies that $\Lambda M \geq 1$.

\paragraph{Vector spaces with an ordered basis.} Let $V$ be a vector space over $\C$ with basis $\set{v_i}{i \in \mathbb{I}}$. The symbol $\mathrm{L}(V)$ denotes the space of linear endomorphisms of $V$. The support of $A\in \mathrm{L}(V)$ is the set of indices
\[
\supp A := \set{(i,j)\in \mathbb{I} \times \mathbb{I}}{a_{i,j} \neq 0}
\] 
which correspond to the non-vanishing coefficients of the matrix $(a_{i,j})$ representing $A$ with respect to the basis $\set{v_i}{i \in \mathbb{I}}$, with the standard convention that $a_{i,j}$ is the coefficient of $v_i$ in the decomposition of $Av_j$. It is useful to see the support of a linear endomorphism as a relation on $\mathbb{I}$, so that
\begin{equation}
\label{contro}
\supp AB \subset (\supp B) \circ (\supp A), \qquad \forall A,B\in \mathrm{L}(V),
\end{equation}
where $R\circ S$ denotes the usual composition of two relations
\[
R \circ S := \set{(i,j) \in \mathbb{I}\times \mathbb{I}}{\exists k\in \mathbb{I} \mbox{ such that } (i,k) \in S \mbox{ and } (k,j)\in R}.
\]
We also recall that the image of an element $i\in \mathbb{I}$, respectively of a subset $E\subset \mathbb{I}$, by a relation $R\subset \mathbb{I}\times \mathbb{I}$ is the following subset of $\mathbb{I}$
\[
R(i) := \set{j\in \mathbb{I}}{(i,j)\in R}, \quad \mbox{resp.} \;\;
R(E) := \bigcup_{i\in E} R(i) = \set{j \in \mathbb{I}}{\exists i\in E \mbox{ such that } (i,j)\in R}.
\]
When dealing with relations which are graphs of self-maps of $\mathbb{I}$, the notion of image and of composition reduce to the ones for maps. Moreover,
\[
R(S(E)) = (R\circ S) (E). 
\]

An endomorphism $A\in \mathrm{L}(V)$ is said to be diagonal if it is supported in the diagonal of $\mathbb{I}\times \mathbb{I}$. It is said to be a permutation automorphism if for every $i\in \mathbb{I}$ it maps $v_i$ into $v_{\sigma(i)}$, where $\sigma$ is a permutation of $\mathbb{I}$.
Assume that $\mathbb{I}$ is endowed with a partial order relation
\[
\mathbb{P} = \set{(i,j)\in \mathbb{I}}{i\leq j}.
\]
Then the endomorphism $A\in \mathrm{L}(V)$ is said to be upper triangular if the corresponding matrix is upper triangular with respect to this order, meaning that
\[
\supp A \subset \mathbb{P}.
\]
It is said to be strictly upper triangular if
\[
\supp A \subset \mathbb{P}_* := \set{(i,j)\in \mathbb{I}\times \mathbb{I}}{i<j}.
\] 
Since 
\[
\mathbb{P}\circ \mathbb{P}\subset \mathbb{P}, \qquad \mathbb{P}\circ \mathbb{P}_* \subset \mathbb{P}_*, \qquad \mathbb{P}_*\circ \mathbb{P} \subset \mathbb{P}_*,
\] 
the inclusion (\ref{contro}) implies that the composition of upper triangular endomorphisms is upper triangular, and strictly upper triangular if at least one of the factors is so. When $V=\C^d$ is equipped with its standard basis
\begin{equation}
\label{defnJ}
\set{e_i}{i\in \mathbb{J}}, \quad \mathbb{J}:= \{1,2,\dots,d\},
\end{equation}
and $\mathbb{J}\subset \N$ has the natural order, we find the standard definitions.

\paragraph{Upper triangular holomorphic maps.} Consider the flag
\[
(0) = E_0 \subset E_1 \subset E_2 \subset \dots \subset E_{d-1} \subset E_d = \C^d, \quad \mbox{where } E_j := \Span\{e_1,\ldots,e_j\}.
\]
A holomorphic map $f\colon\C^d\rightarrow \C^d$ is said to be upper triangular if it preserves this flag, that is $f(E_j) \subseteq E_j$ for every $j=0,\dots,d$. It is said to be strictly upper triangular if $f(E_j) \subseteq E_{j-1}$ for every $j=1,\dots,d$. A holomorphic map $f\colon\C^d \rightarrow \C^d$ is upper triangular (respectively, strictly upper triangular) if and and only if $f(0)=0$ and for every $j=1\dots,d$ the $j$-the component of $f$ depends only on the variables $z_j,\dots ,z_d$ (respectively, $z_{j+1},\dots, z_d$). For linear endomorphisms of $\C^d$, the notion of (strict) upper tringularity is the standard one. Again, the composition of upper triangular maps is still upper triangular, and it is strictly upper triangular if at least one of the maps is so. 

\paragraph{The conjugacy operator on the space of homogeneous polynomial maps.} Let $\mathscr{H}_k$ be the vector space of homogeneous polynomial maps $p\colon \C^d \rightarrow \C^d$ of degree $k$. We endow  $\mathscr{H}_k$ with its standard basis
\[
\set{z^{\alpha} e_i}{ (\alpha,i) \in \mathbb{H}_k},
\]
where
\[
\mathbb{A}_k := \set{\alpha\in \N^d}{|\alpha|=k}, \qquad \mathbb{H}_k := \mathbb{A}_k \times \mathbb{J},
\]
$|\alpha|:=\alpha_1 + \dots + \alpha_d$ is the degree of the multi-index
$\alpha$, and $\mathbb{J}$ is as in (\ref{defnJ}).
If we set
\[
\mathbb{T}_k := \set{(\alpha,i)\in \mathbb{H}_k }{
\alpha_j = 0 \; \forall j\leq i}, 
  \quad \mathscr{T}_k := \mathrm{span}\, \set{z^{\alpha}
    e_i}{(\alpha,i) \in \mathbb{T}_k} ,
\]
then a polynomial map 
\[
p\colon\C^d \rightarrow \C^d, \quad p = \sum_{k=1}^m p_k, \quad p_k \in \mathscr{H}_k,
\]
is strictly upper triangular if and only if each $p_k$ belongs to
$\mathscr{T}_k$.

If $L$ is a linear automorphism of $\C^d$, we consider the conjugacy operator
\[
\mathscr{A}_L : \mathscr{H}_k \rightarrow \mathscr{H}_k, \quad \mathscr{A}_L p := L^{-1} \circ p \circ L, \quad \forall p\in \mathscr{H}_k.
\]
The conjugacy operator $\mathscr{A}_L$ depends controvariantly on the automorphism $L$:
\[
\mathscr{A}_{LM} = \mathscr{A}_M \mathscr{A}_L, \qquad \forall L,M \in \mathrm{GL}(\C^d).
\]

Let $L$ be an upper triangular linear automorphism of $\C^d$, let $(\ell_{i,j})$ be the upper triangular matrix which represents it and let $\alpha\in \mathbb{A}_k$. Then the formula
\[
(Lz)^{\alpha} = \Bigl( \ell_{1,1} z_1 + \sum_{j=2}^d \ell_{1,j} z_j \Bigr)^{\alpha_1} \Bigl( \ell_{2,2} z_2 + \sum_{j=3}^d \ell_{2,j} z_j \Bigr)^{\alpha_2} \dots \bigl( \ell_{d-1,d-1} z_{d-1} + \ell_{d-1,d} z_d\bigr)^{\alpha_{d-1}} \bigl(\ell_{d,d} z_d\bigr)^{\alpha_d} 
\]
implies that
\begin{equation}
\label{Lz}
(Lz)^{\alpha} = \lambda^{\alpha} z^{\alpha} + q(z), \qquad \mbox{where } \lambda := (\ell_{1,1},\ell_{2,2}, \dots, \ell_{d,d}), 
\end{equation}
and where the degree of $q$ in the variables $z_1,\dots,z_m$ is strictly less that $\sum_{h=1}^m \alpha_h$, for every $1\leq m \leq d$. 

Decompose $L$ as $L=D+S$ with $D=\mathrm{Diag}(\lambda)$ diagonal, $\lambda\in \C^d$, and $S$ strictly upper triangular. Then $L^{-1} = D^{-1} + T$, with $T$ strictly upper triangular, and for every $p\in \mathscr{H}_k$ we have
\[
\mathscr{A}_L p = L^{-1} \circ p \circ L = D^{-1} \circ p \circ L + T \circ p \circ L.
\]
If $p(z) = z^{\alpha} e_i$ is an element of the standard basis of $\mathscr{H}_k$, then (\ref{Lz}) implies that
\[
D^{-1} \circ p \circ L \in \lambda^{\alpha - e_i} z^{\alpha} e_i + \mathrm{span} \, \set{z^{\beta} e_i}{\beta\in \mathbb{A}_k \mbox{ is such that } \sum_{h=1}^m \beta_h < \sum_{h=1}^m \alpha_h \mbox{ for every } 1\leq m \leq d}.
\]
Moreover, since $T$ is strictly upper triangular,
\[
T \circ p \circ L = (Lz)^{\alpha} T e_i \in (\lambda^{\alpha} z^{\alpha} + q(z)) E_{i-1},
\]
so the $k$-homogeneous polynomial map $T \circ p \circ L$ belongs to the subspace
\[
\set{z^{\beta} e_j}{1\leq j < i \mbox{ and } \beta\in \mathbb{A}_k \mbox{ such that }  
\sum_{h=1}^m \beta_h < \sum_{h=1}^m \alpha_h \mbox{ for every } 1\leq m \leq d}.
\]
These considerations suggest to introduce the following partial order on the set $\mathbb{H}_k$:
\begin{equation}
\label{order}
\begin{split}
\mathbb{P}_k := \set{ \bigl( (\alpha,i), (\beta,j) \bigr) \in \mathbb{H}_k \times \mathbb{H}_k }{(\alpha,i) \leq (\beta,j)}, \mbox{ where} \\
(\alpha,i) \leq (\beta,j) \quad \iff \quad i \leq j \quad \mbox{and} \quad \sum_{h=1}^m \alpha_h \leq \sum_{h=1}^m \beta_h \quad \mbox{for every } 1\leq m\leq d.
\end{split} \end{equation}
This is the direct product between the standard order of $\mathbb{J}$ and the following order on $\mathbb{A}_k$: $\alpha\leq \beta$ if and only if for every $1\leq m \leq d$ the degree of $z^{\alpha}$ with respect to the variables $z_1,\dots,z_m$ does not exceed that of $z^{\beta}$. The above considerations imply that if $L\in \mathrm{GL}(\C^d)$ is upper triangular and $\lambda\in \C^d$ is the vector of its diagonal entries, then
\[
\mathscr{A}_{L} z^{\alpha} e_i \in \lambda^{\alpha-e_i} z^{\alpha} e_i + \mathrm{span} \set{z^{\beta}e_j}{(\beta,j) < (\alpha,i) },
\]
for every $(\alpha,i)\in \mathbb{H}_k$.

If $\sigma$ belongs to the group $\mathfrak{S}(\mathbb{J})$ of permutations of the set $\mathbb{J}$, we denote by $U_{\sigma}$ the permutation automorphism of $\C^d$ such that 
\[
U_{\sigma} e_j = e_{\sigma(j)}, \quad \forall j\in \mathbb{J}. 
\]
Conjugation by the inverse of $U_{\sigma}$ produces the identity
\[
\mathscr{A}_{U_{\sigma}^{-1}} (z^{\alpha} e_i) = z^{\alpha\circ \sigma^{-1}} e_{\sigma(i)}, \quad \forall (\alpha,i) \in \mathbb{H}_k,
\]
which defines the following left action of the group $\mathfrak{S}(\mathbb{J})$ on the set $\mathbb{H}_k$:
\[
\mathfrak{S}(\mathbb{J}) \times \mathbb{H}_k \rightarrow \mathbb{H}_k, \quad \bigl(\sigma,(\alpha,i)\bigr) \mapsto \sigma\cdot (\alpha,i) := (\alpha\circ \sigma^{-1},\sigma(i)).
\]
Notice that the support of the endomorphism $\supp \mathscr{A}_{U_{\sigma}}$, thought as a relation on $\mathbb{H}_k$, acts on subsets of $\mathbb{H}_k$ as
\begin{equation}
\label{perm}
(\supp \mathscr{A}_{U_{\sigma}})(E) = \sigma \cdot E, \qquad \forall E\subset \mathbb{H}_k.
\end{equation}

We can summarize what we have seen so far into the following lemma, which gathers some of the properties of $L$ which are inherited by the conjugacy operator $\mathcal{A}_L$.

\begin{lem}
\label{tria}
Endow $\C^d$ and $\mathscr{H}_k$ with the ordered basis $\set{e_i}{i\in \mathbb{J}}$ and $\set{z^{\alpha} e_i}{(\alpha,i)\in \mathbb{H}_k}$, respectively. Let $L$ be a linear automorphism of $\C^d$.
\begin{enumerate}
\item If $L$ is diagonal then $\mathcal{A}_L$ is diagonal and
\[
\mathscr{A}_{L} z^{\alpha} e_i \in \lambda^{\alpha-e_i} z^{\alpha} e_i, \qquad \forall (\alpha,i) \in \mathbb{H}_k,
\]
where $\lambda\in \C^d$ is the vector of diagonal entries of $L$.
\item If $L$ is upper triangular then $\mathcal{A}_L$ is upper triangular and
\[
\mathscr{A}_{L} z^{\alpha} e_i \in \lambda^{\alpha-e_i} z^{\alpha} e_i + \mathrm{span} \set{z^{\beta}e_j}{(\beta,j) < (\alpha,i) }, \qquad \forall (\alpha,i) \in \mathbb{H}_k,
\]
where $\lambda\in \C^d$ is the vector of diagonal entries of $L$.
\item if $L$ is a permutation automorphism then $\mathcal{A}_L$ is a permutation automorphism and, setting $L=U_{\sigma}^{-1}$ with $\sigma \in \mathfrak{S}(\mathbb{J})$,
\[
\mathscr{A}_{U_{\sigma}^{-1}}  z^{\alpha} e_i = z^{\beta} e_j, \qquad \mbox{with } (\beta,j) := \sigma \cdot (\alpha,i) = (\alpha\circ \sigma^{-1},\sigma(i)).
\]
\end{enumerate}
\end{lem} 

\paragraph{The case of degree 2.} We now specialize the analysis to the case of degree 2, and we drop the subscript $2$ in the notation, thus setting $\mathscr{H}=\mathscr{H}_2$, $\mathbb{A}=\mathbb{A}_2$, $\mathbb{H}=\mathbb{H}_2$, $\mathbb{T}=\mathbb{T}_2$, $\mathbb{P}=\mathbb{P}_2$.
 
Given $s\in \mathbb{H}$, we set
\[
s^{\uparrow} := \set{r\in \mathbb{H}}{r\geq s},
\]
and, for $E\subset \mathbb{H}$,
\[
E^{\uparrow} := \bigcup_{s\in E} s^{\uparrow},
\]
where we are using the order defined in (\ref{order}). The mapping $E\mapsto E^{\uparrow}$ is the closure operator associated to the left order topology on the partially ordered set $(\mathbb{H},\leq)$ (see e.g.\ \cite{ss78}). Notice that 
\[
\mathbb{T} = \bigl( (\mathbb{V}^c)^{\uparrow} \bigr)^c,
\]
where
\[
\mathbb{V}:= \set{(\alpha,i)\in \mathbb{H}}{\alpha_i = 0},
\]
that is, $\T$ is the interior part of the $\mathfrak{S}(\mathbb{J})$-invariant set $\mathbb{V}$ in the lower order topology. The following relation on $\mathbb{H}$ turns out to be useful in order to give estimates for the conjugacy operator by upper triangular automorphisms: 
\[
\mathbb{W} := \set{ (s,t)\in \mathbb{H} \times \mathbb{H} }{s\in \mathbb{T}^c \mbox{ and } \exists r\in \mathbb{V} \mbox{ such that } s\leq r \leq t} = \bigcup_{s\in \T^c} \{s\} \times ( \mathbb{V} \cap s^{\uparrow} )^{\uparrow}.
\]
Since $\T^c = (\T^c)^{\uparrow}$, we have
\[
\mathbb{W} \subset \T^c \times \T^c.
\]
The image of a subset $E\subset \mathbb{H}$ by the relation $\mathbb{W}$ is the subset
\begin{equation}
\label{stra}
\begin{split}
\mathbb{W}(E) = \bigcup_{s\in E} \mathbb{W}(s) =   \bigcup_{s\in E\setminus \T} ( \mathbb{V} \cap s^{\uparrow} )^{\uparrow} = \Bigl( \bigcup_{s\in E\setminus \T}  ( \mathbb{V} \cap s^{\uparrow} ) \Bigr)^{\uparrow} \\ = \Bigl( \mathbb{V} \cap \bigcup_{s\in E\setminus \T} s^{\uparrow} \Bigr)^{\uparrow} = \Bigl( \mathbb{V} \cap (E\setminus \T)^{\uparrow} \Bigr)^{\uparrow} \subset E^{\uparrow},
\end{split}
\end{equation}
where in the third equality we have used the fact that $E\mapsto E^{\uparrow}$ is a closure operator. We claim that
\begin{equation}
\label{idenpot}
\mathbb{W} \circ \mathbb{W} = \mathbb{W}.
\end{equation}
Indeed, for every $s\in \mathbb{H}$, by taking the closures of the trivial inclusions
\begin{eqnarray*}
\mathbb{V} \cap \bigl( \mathbb{V} \cap s^{\uparrow} \bigr)^{\uparrow} & \subset & \bigl (\mathbb{V} \cap s^{\uparrow} \bigr)^{\uparrow}, \\
\mathbb{V} \cap s^{\uparrow} & \subset  & \mathbb{V} \cap \bigl (\mathbb{V} \cap s^{\uparrow} \bigr)^{\uparrow},
\end{eqnarray*}
we obtain the identity
\[
\Bigl( \mathbb{V} \cap \bigl( \mathbb{V} \cap s^{\uparrow}  \bigr)^{\uparrow} \Bigr)^{\uparrow} = \bigl( \mathbb{V} \cap s^{\uparrow} \bigr)^{\uparrow} = \mathbb{W}(s),
\]
which implies
\[
\mathbb{W} \circ \mathbb{W} (s) = \mathbb{W} (\mathbb{W}(s)) = \Bigl( \mathbb{V} \cap \bigl( \mathbb{W}(s) \setminus \T \bigr)^{\uparrow} \Bigr)^{\uparrow} = \Bigl( \mathbb{V} \cap \bigl( \mathbb{W}(s)\bigr)^{\uparrow} \Bigr)^{\uparrow} = \Bigl( \mathbb{V} \cap \bigl( \mathbb{V} \cap s^{\uparrow}  \bigr)^{\uparrow} \Bigr)^{\uparrow} = \mathbb{W}(s).
\]

Let 
\[
\mathscr{Q} : \mathscr{H} \rightarrow \mathscr{H},
\]
be the linear projector onto the subspace spanned by $z^{\alpha}e_i$ for $(\alpha,i)$ in $\T^c$, along the subspace spanned by $z^{\alpha}e_i$ for $(\alpha,i)$ in $\T$.
We notice that if $L$ is upper triangular, then $\mathscr{T} = \ker \mathscr{Q}$ is $\mathscr{A}_L$-invariant, or equivalently
\[
\mathscr{Q} \mathscr{A}_L =  \mathscr{Q} \mathscr{A}_L \mathscr{Q}.
\]
Let $(L_n)$ be a $(\Lambda,M)$-pinched sequence of upper triangular linear automorphisms of $\C^d$. We fix some $n\in \N$ and consider the conjugacy operator $\mathscr{A}_{L_{n,0}}$ on the space $\mathscr{H}$ of homogenous polynomial maps of degree 2. By the inequality
\[
\| \mathscr{A}_{L_{n,0}} u \| = \| L_{n,0}^{-1} \circ u \circ L_{n,0} \| \leq \|L_{n,0}^{-1}\| \|L_{n,0}\|^2 \|u\| , \qquad \forall u\in \mathscr{H},
\]
the pinching estimate (\ref{est}) implies the following upper bound on the operator norm of $\mathscr{A}_{L_{n,0}}$:
\begin{equation}
\label{norma}
\|\mathscr{A}_{L_{n,0}}\| \leq C^3 (\Lambda^2 M)^n, \qquad \forall n\in \N.
\end{equation}
Moreover, we know from Lemma \ref{tria} (ii) that the operator $\mathscr{A}_{L_{n,0}}$ is upper triangular, meaning that its support is contained in $\mathbb{P}$. The following lemma says that  the coefficients of $\mathscr{A}_{L_{n,0}}$ which belong to the set $(\mathbb{T}^c \times \mathbb{T}^c)\setminus \mathbb{W}$ have an upper bound which is better than the one implied by (\ref{norma}):

\begin{lem}
\label{decomp}
Let $(L_n)$ be a $(\Lambda,M)$-pinched sequence of upper triangular linear automorphisms of $\C^d$. Consider the decomposition 
\[
\mathscr{Q} \mathscr{A}_{L_{n,0}} =  \mathscr{Q} \mathscr{A}_{L_{n,0}} \mathscr{Q} = \mathscr{M}^0 + \mathscr{M}^1
\]
with
\[
\supp \mathscr{M}^0 \subset (\T^c \times \T^c) \setminus \mathbb{W}, \qquad
\supp \mathscr{M}^1 \subset \mathbb{W}.
\]
Then there exists positive numbers $C_0$ and $C_1$ such that
\begin{eqnarray}
\label{buona}
\| \mathscr{M}^0 \| &\leq &  C_0 \,n^{3(d-1)} \,\Lambda^n,  \\
\label{cattiva}
\| \mathscr{M}^1 \| &\leq & C_1 \, (\Lambda^2 M)^n,
\end{eqnarray}
for every $n\in \N$.
\end{lem}

\begin{proof}
Estimate (\ref{cattiva}) is an immediate consequence of (\ref{norma}), so the estimate to be proved is (\ref{buona}). 
Let 
\[
\lambda_h = (\lambda_h(1),\dots,\lambda_h(d))\in \C^d
\]
be the vector of diagonal entries of $L_h$. By Lemma \ref{tria} (ii), $\mathscr{A}_{L_h}$ decomposes as the sum 
\[
\mathscr{A}_{L_h} = \mathscr{B}_h^0 + \mathscr{B}_h^1,
\]
where $\mathscr{B}_h^0$ is diagonal and $\mathscr{B}_h^1$ is strictly upper triangular, with coefficients 
\begin{equation}
\label{la0}
b_h^0 \bigl((\alpha,i),(\beta,j)\bigr) = \langle \mathscr{B}_h^0 z^{\beta} e_j, z^{\alpha} e_i \rangle = \left\{ \begin{array}{ll} \lambda_h^{\alpha-e_i}, & \mbox{if } (\alpha,i) = (\beta,j), \\ 0 , & \mbox{otherwise}, \end{array} \right. 
\end{equation}
respectively
\begin{equation}
\label{la1}
b_h^1 \bigl((\alpha,i),(\beta,j)\bigr) = \langle \mathscr{B}_h^1 z^{\beta} e_j, z^{\alpha} e_i \rangle = \left\{ \begin{array}{ll} \langle \mathscr{A}_{L_h} z^{\beta} e_j, z^{\alpha} e_i \rangle, & \mbox{if } (\alpha,i) < (\beta,j), \\ 0 , & \mbox{otherwise}. \end{array} \right. 
\end{equation}
Here $\langle \cdot, \cdot \rangle$ denotes the inner product on $\mathscr{H}$ which makes $\set{z^{\alpha}e_i}{(\alpha,i)\in \mathbb{H}}$ an orthonormal basis.
Since the sequences $(L_n)$ and $(L_n^{-1})$ are bounded, 
\begin{equation}
\label{eA}
| b_h^1 \bigl((\alpha,i),(\beta,j)\bigr) | \leq c, \qquad \forall h\in \N,
\end{equation}
for some $c>0$. By the pinching condition (\ref{est}), if $(\alpha,i)$ belongs to $\mathbb{V}^c$, that is $\alpha = e_i + e_j$ for some $j$ then
\begin{equation}
\label{eB}
| b_{k,h}^0 \bigl((\alpha,i),(\alpha,i)\bigr) | = \left| \frac{\lambda_{k,h} (i) \lambda_{k,h} (j)}{\lambda_{k,h} (i)} \right| = 
 | \lambda_{k,h} (j) | \leq \|L_{k,h}\| \leq C\, \Lambda^{k-h},
\end{equation}
for every $0\leq h \leq k$. 

By using the above decomposition of each $\mathscr{A}_{L_h}$, we find the following formula for the coefficients of $\mathscr{A}_{L_{n,0}}$:
\begin{eqnarray*}
\langle \mathscr{A}_{L_{n,0}} z^{\beta} e_j, z^{\alpha} e_i \rangle &=& \langle \mathscr{A}_{L_0} \mathscr{A}_{L_1} \dots  \mathscr{A}_{L_{n-1}} z^{\beta} e_j, z^{\alpha} e_i \rangle \\
 &=&  \langle (\mathscr{B}_0^0 + \mathscr{B}_0^1) (\mathscr{B}_1^0 + \mathscr{B}_1^1)\dots  (\mathscr{B}_{n-1}^0 + \mathscr{B}_{n-1}^1) z^{\beta} e_j, z^{\alpha} e_i \rangle \\ &=& \sum_{\epsilon\in \mathbbm{2}^n} \langle \mathscr{B}_0^{\epsilon_0}  \mathscr{B}_1^{\epsilon_1} \dots  \mathscr{B}_{n-1}^{\epsilon_{n-1}} z^{\beta} e_j, z^{\alpha} e_i \rangle \\
 &=& \sum_{\epsilon\in \mathbbm{2}^n} \sum_{\substack{s \in \mathbb{H}^{n+1}\\ s_0 = (\alpha,i), \; s_n = (\beta,j)}} b_0^{\epsilon_0} (s_0,s_1)  b_1^{\epsilon_1} (s_1,s_2) \dots b_{n-1}^{\epsilon_{n-1}} (s_{n-1},s_n).
 \end{eqnarray*}
Since the $\mathscr{B}_h^0$'s are diagonal and the $\mathscr{B}_h^1$'s are strictly upper triangular, the non-zero terms in the last sum correspond to sequences $s=(s_0,\dots,s_n)\in \mathbb{H}^{n+1}$ and $\epsilon=(\epsilon_0,\dots,\epsilon_{n-1}) \in \mathbbm{2}^n$  such that
\[
(\alpha,i) = s_0 \leq s_1 \leq \dots \leq s_{n-1} \leq s_n = (\beta,j),
\]
where 
\[
s_h < s_{h+1} \quad \iff \quad \epsilon(h)=1.
\]
Therefore, the above formula can be rewritten as
\[
\langle \mathscr{A}_{L_{n,0}} z^{\beta} e_j, z^{\alpha} e_i \rangle = \hskip-1cm  \sum_{\substack{k\in \N \\ (\alpha,i) = t_0 < t_1 < \dots < t_k = (\beta,j) \\ 0 \leq n_1 < n_2 < \dots < n_k < n}} \hskip-1cm  b_{n_1,0}^0(t_0) b_{n_1}^1(t_0,t_1) b_{n_2,n_1+1}^0(t_1) \dots b_{n_k}^1(t_{k-1},t_k) b_{n,n_k+1}^0(t_k).
\]
The height of the partially ordered set $(\mathbb{H},\leq)$, that is the maximal cardinality of chains in $\mathbb{H}$, is $3d-2$. It follows that the natural number $k$ in the above sum is at most $3(d-1)$. If we also assume that
\begin{equation}
\label{dovedove}
\bigl( (\alpha,i) , (\beta,j) \bigr) \in (\T^c \times \T^c) \setminus \mathbb{W},
\end{equation}
then the definition of the relation $\mathbb{W}$ implies that in every chain 
\[
(\alpha,i) = t_0 < t_1 < \dots < t_k = (\beta,j)
\]
each index $t_h$, $0\leq h \leq k$, belongs to $\mathbb{V}^c$. Thus, under the assumption (\ref{dovedove}), the estimates (\ref{eA}) and (\ref{eB}) imply the bound
\[
\bigl| \langle \mathscr{A}_{L_{n,0}} z^{\beta} e_j, z^{\alpha} e_i \rangle \bigr| \leq\hskip-1cm   \sum_{\substack{0\leq k\leq 3(d-1)\\ (\alpha,i) = t_0 < t_1 < \dots < t_k = (\beta,j) \\ 0 \leq n_1 < n_2 < \dots < n_k < n}} \hskip-1cm \bigl( C\, \Lambda^{n_1} \bigr) c \bigl( C\, \Lambda^{n_2-n_1-1} \bigr)  \dots c \bigl( C\, \Lambda^{n-n_k-1} \bigr) = \Lambda^n \hskip-1cm   \sum_{\substack{0\leq k\leq 3(d-1) \\ (\alpha,i) = t_0 < t_1 < \dots < t_k = (\beta,j) \\ 0 \leq n_1 < n_2 < \dots < n_k < n}} \hskip-1cm c^k C^{k+1} \Lambda^{-k}.
\]
Since $\mathbb{H}$ is a finite set and the set of $k$-uples $(n_1,\dots,n_k)\in \N^k$ such that $0\leq n_1 <  \dots < n_k < n$ has 
\[
\binom{n}{k} \leq e \, n^k \leq e\, n^{3(d-1)}, 
\]
elements, the last sum is bounded from above by $c_1 n^{3(d-1)}$, for some constant $c_1$, and we find
\[
\bigl| \langle \mathscr{A}_{L_{n,0}} z^{\beta} e_j, z^{\alpha} e_i \rangle \bigr| \leq c_1 n^{3(d-1)} \Lambda^n, \quad \forall \bigl( (\alpha,i) , (\beta,j) \bigr) \in (\T^c \times \T^c) \setminus \mathbb{W}.
\]
The estimate (\ref{buona}) follows. 
\end{proof}

\section{A nilpotency result}

We continue to work with the vector space $\mathscr{H}_2$, equipped with its standard basis indexed by $\mathbb{H}_2$, and we drop the subscript 2 from the notation, as in the last part of the previous section.
We denote by $\kappa_h\in \mathfrak{S}(\mathbb{J})$, for $1\leq h \leq d$, the cyclic permutation
\[
\kappa_h := (1,2,\dots,h)
\]
over the first $h$ elements.
We consider the finite sequence of elements of $\mathbb{J}$
\[
n = n^{(d)} = \bigl( n_1, n_2, \dots, n_{\binom{d}{2}} \bigr)
\]
defined as
\[
\begin{array}{cr}
\emptyset & \quad\mbox{for } d=1, \\
(2) & \quad\mbox{for } d=2, \\
(2,2,3) & \quad\mbox{for } d=3, \\
(2,3,2,2,3,4) & \quad\mbox{for } d=4, \\
(2,3,4,2,3,2,2,3,4,5) & \quad\mbox{for } d=5,
\end{array} 
\]
and in general $n^{(d)}$ is obtained by juxtaposition of the strings $(2,3,\dots,d-1)$, $n^{(d-1)}$ and $(d)$. The following lemma is the main combinatorial tool which leads to the estimates of Section \ref{2jetssec}. 

\begin{lem}
\label{combi}
The finite sequence of elements of the group of permutations $\mathfrak{S}(\mathbb{J})$ 
\[
\Bigl(\sigma_1, \sigma_2, \dots, \sigma_{\binom{d}{2}} \Bigr) := \Bigl( \kappa_{n_1}, \kappa_{n_2}, \dots, \kappa_{n_{\binom{d}{2}}} \Bigr)
\]
satisfies 
\begin{equation}
\label{latesi}
\mathbb{W}\, \sigma_{\binom{d}{2}} \, \mathbb{W} \, \sigma_{\binom{d}{2}-1} \, \dots \, \sigma_2 \, \mathbb{W} \, \sigma_1 \, \mathbb{W}  = \emptyset.
\end{equation}
\end{lem}

\begin{proof}
When $d=1$, $\mathbb{V}$ is empty and so is $\mathbb{W}$, hence the identity (\ref{latesi}) trivially holds for the empty string. When $d=2$, we have
\[
\mathbb{T} = \{ (2e_2,1)\}, \quad \mathbb{V}=\{(2e_2,1),(2e_1,2)\}.
\]
Since the set $\mathbb{V}\setminus \mathbb{T} = \{(2e_1,2)\}$ contains only the maximum of $\mathbb{H}$ by (\ref{stra}),  there holds
\[
\mathbb{W}(\mathbb{H}) = \{(2e_1,2)\}.
\]
Therefore
\[
\mathbb{W} (\kappa_2 ( \mathbb{W}(\mathbb{H}))) = \mathbb{W}(2e_2,1) = \emptyset,
\]
and the string $(\kappa_2)$ satisfies (\ref{latesi}).

Let $d\geq 3$ and consider the closed set
\[
Y:= \set{(\alpha,i)\in \mathbb{H}}{\alpha_d=0}.
\]
If $1\leq i \leq d$, then
\begin{equation}
\label{cuno}
\bigl( \mathbb{A} \times \{i\} \setminus \T \bigr)^{\uparrow} = \set{ (\alpha,i)\in \mathbb{H}}{\exists j\leq i \mbox{ such that } \alpha_j\neq 0}^{\uparrow} = (e_i+e_d,i)^{\uparrow},
\end{equation}
because $(e_i+e_d,1)$ is the minimum of $\mathbb{A} \times \{i\} \setminus \T$. Since $(e_1+e_{d-1},1)$ is the minimum of $\mathbb{A}\times \{1\} \setminus \{(e_1+e_d,1)\} \setminus \T$, we also have
\begin{equation}
\label{cdue}
\bigl( \mathbb{A} \times \{1\} \setminus \{(e_1+e_d,1)\} \setminus \T \bigr)^{\uparrow} = (e_1+e_{d-1},1)^{\uparrow}.
\end{equation}
For $1<i<d$, there holds
\[
(e_i+e_d,i)^{\uparrow} = \{(e_i+e_d,i)\} \cup (e_{i-1}+e_d,i)^{\uparrow} \cup (e_i+e_{d-1},i)^{\uparrow} \cup (e_i+e_d,i+1)^{\uparrow}.
\]
Since $(e_i+e_d,i)$ does not belong to $\mathbb{V}$ and $(e_i+e_{d-1},i)^{\uparrow}\subset Y$, the above identity and (\ref{cuno}) imply
\[
\bigl( \mathbb{A} \times \{i\} \setminus \T \bigr)^{\uparrow} \cap \mathbb{V} \subset Y \cup (e_{i-1}+e_d,i)^{\uparrow} \cup (e_i+e_d,i+1)^{\uparrow}.
\]
The set on the right is closed, so by (\ref{stra}),
\begin{equation}
\label{ctre}
\mathbb{W}\bigl(\mathbb{A}\times \{i\}\bigr) = \Bigl( \bigl( \mathbb{A}\times \{i\} \setminus \T \bigr)^{\uparrow} \cap \mathbb{V} \Bigr)^{\uparrow} \subset Y \cup (e_{i-1}+e_d,i)^{\uparrow} \cup (e_i+e_d,i+1)^{\uparrow}, \quad \mbox{for } 1< i < d.
\end{equation}
Similarly, for $i=1$ we get
\begin{equation}
\label{cquattro}
\mathbb{W}\bigl(\mathbb{A}\times \{1\}\bigr)  \subset Y \cup (e_1+e_d,2)^{\uparrow}.
\end{equation}
Since $\mathbb{A}\times \{d\} \cap \T=\emptyset$ and $\mathbb{A}\times \{d\} \cap \mathbb{V}$ is closed,  for $i=d$ we have 
\begin{equation}
\label{ccinque}
\mathbb{W}\bigl(\mathbb{A}\times \{d\}\bigr) = \Bigl( \bigl( \mathbb{A}\times \{d\} \setminus \T \bigr)^{\uparrow} \cap \mathbb{V} \Bigr)^{\uparrow} = \Bigl( \bigl( \mathbb{A}\times \{d\} \bigr)^{\uparrow} \cap \mathbb{V} \Bigr)^{\uparrow} = \mathbb{A}\times \{d\} \cap \mathbb{V} \subset Y.
\end{equation}
Finally, from (\ref{cdue}) we also find
\begin{equation}
\label{csei}
\mathbb{W}\bigl( \mathbb{A}\times \{1\} \setminus \{(e_1+e_d,1)\}\bigr) \subset Y.
\end{equation}
We claim that 
\begin{equation}
\label{claim}
\mathbb{W}\circ \kappa_{d-1} \circ \mathbb{W} \circ \dots \circ \kappa_2 \circ \mathbb{W} (\mathbb{H}) \subset Y.
\end{equation}
In order to prove this claim, we introduce the sets
\[
X_1 := \mathbb{H}, \qquad X_r := Y \cup \set{(\alpha,i)\in \mathbb{H}}{i>r} \cup \bigl( \mathbb{A}\times \{1\} \setminus \{ (e_1+e_d,1)\} \bigr), \quad \mbox{for } 2\leq r \leq d-1,
\]
and we show that 
\begin{equation}
\label{csette}
\kappa_{r+1} ( \mathbb{W}(X_r)) \subset X_{r+1}, \qquad \mbox{for } 1\leq r \leq d-2.
\end{equation}
We get from (\ref{ctre}), (\ref{cquattro}) and (\ref{ccinque}),
\begin{eqnarray*}
\mathbb{W}(\mathbb{H}) = \bigcup_{1\leq i \leq d} \mathbb{W}(\mathbb{A}\times \{i\}) \subset Y \cup \bigcup_{1\leq i < d} (e_i+e_d,i+1)^{\uparrow} \\
\subset Y \cup \set{(\alpha,i) \in \mathbb{H}}{i\geq 3} \cup (e_1+e_d,2)^{\uparrow} =  Y \cup \set{(\alpha,i) \in \mathbb{H}}{i\geq 3} \cup \{(e_1+e_d,2)\}.
\end{eqnarray*}
Since the first two sets in the last union are invariant with respect to $\kappa_2$ and since $\kappa_2(d)=d$ because $d\geq 3$, 
\[
\kappa_2 (\mathbb{W}(X_1)) = \kappa_2 (\mathbb{W}(\mathbb{H})) \subset Y \cup \set{(\alpha,i) \in \mathbb{H}}{i\geq 3} \cup \{(e_2+e_d,1)\} \subset X_2,
\]
which proves (\ref{csette}) for $r=1$. Let $2\leq r \leq d-2$ and write
\[
X_r = Y \cup \set{(\alpha,i)\in \mathbb{H}}{i>r+1} \cup \mathbb{A} \times \{r+1\} \cup  \bigl( \mathbb{A}\times \{1\} \setminus \{(e_1+e_d,1)\} \bigr).
\]
The first two sets in the above union are $\mathbb{W}$-invariant. Thanks to (\ref{ctre}) and (\ref{csei}), we have
\begin{equation}
\label{cotto}
\begin{split}
\mathbb{W}(X_r) &\subset Y \cup \set{(\alpha,i)\in \mathbb{H}}{i>r+1} \cup  (e_r+e_d,r+1)^{\uparrow}  \\ &= Y \cup \set{(\alpha,i)\in \mathbb{H}}{i>r+1} \cup \bigl( (e_r+e_d,r+1)^{\uparrow} \cap \mathbb{A}\times \{r+1\} \bigr).
\end{split}
\end{equation}
Since
\[
\kappa_{r+1}^{-1} (e_1+e_d,1) = (e_{r+1}+e_d,r+1) < (e_r+e_d,r+1),
\]
there holds
\[
(e_1+e_d,1) \notin \kappa_{r+1} \bigl( (e_r+e_d,r+1)^{\uparrow} \bigr),
\]
and
\[
\kappa_{r+1} \bigl( (e_r+e_d,r+1)^{\uparrow} \cap \mathbb{A}\times \{r+1\} \bigr) \subset \mathbb{A}\times \{1\} \setminus \{(e_1+e_d,1)\}.
\]
Together with the fact that the first two sets on the right-hand side of (\ref{cotto}) are $\kappa_{r+1}$-invariant, we deduce that
\[ 
\kappa_{r+1} (\mathbb{W}(X_r)) \subset Y \cup \set{(\alpha,i)\in \mathbb{H}}{i>r+1} \cup \bigl(\mathbb{A}\times \{1\} \setminus \{(e_1+e_d,1)\} \bigr) = X_{r+1},
\]
which proves (\ref{csette}). Iterating (\ref{csette}), we find
\[
\kappa_{d-1}\circ \mathbb{W} \circ \dots \circ \kappa_2 \circ \mathbb{W}(\mathbb{H}) \subset X_{d-1} = Y \cup \mathbb{A} \times \{d\} \cup \bigl( \mathbb{A}\times \{1\} \setminus \{(e_1+e_d,1)\} \bigr).
\]
By (\ref{ccinque}) and (\ref{csei}), a further application of $\mathbb{W}$ proves our claim (\ref{claim}). 

With the purpose of proving the lemma by induction on the dimension $d$, we notice that, making the dependence on $d$ explicit, the relations $\mathbb{W}^{(d)}$ on the sets $\mathbb{H}^{(d)}$ associated to different dimensions are related by
\begin{eqnarray*}
\mathbb{W}^{(d)} (A) \cap \mathbb{H}^{(d-1)} &=& \mathbb{W}^{(d-1)} ( A \cap \mathbb{H}^{(d-1)}), \\
\mathbb{W}^{(d)}(A) \setminus \mathbb{H}^{(d-1)} &\subset& Y^{(d)} \cap \bigl( \mathbb{A}^{(d)} \times \{d\}\bigr),
\end{eqnarray*}
for every $A\subset Y^{(d)}$. Therefore, if we assume that the thesis (\ref{latesi}) holds in dimension $d-1$, since
\[
\sigma_{h+d-2} =  n_{h+d-2}^{(d)} = n_h^{(d-1)} , \qquad \mbox{for } 1\leq h \leq \binom{d-1}{2},   
\]
(under the natural identification of $\mathfrak{S}(\mathbb{J}^{(d-1)})$ as a subgroup of $\mathfrak{S}(\mathbb{J}^{(d)})$), we find
\[
\mathbb{W}\, \sigma_{\binom{d}{2}-1} \, \mathbb{W} \, \dots \, \sigma_{d-1} \, \mathbb{W}  (Y) \subset Y \cap \bigl( \mathbb{A} \times \{d\} \bigr).
\]
Together with (\ref{idenpot}) and (\ref{claim}), this implies that, since $(\sigma_1,\dots,\sigma_{d-2}) = (\kappa_2,\dots,\kappa_{d-1})$,
\[
\mathbb{W}\, \sigma_{\binom{d}{2}-1} \, \mathbb{W} \, \dots \, \sigma_{d-1} \, \mathbb{W} \, \sigma_{d-2} \, \mathbb{W} \, \dots \, \sigma_1 \, \mathbb{W}\, (\mathbb{H}) \subset  Y \cap \bigl( \mathbb{A} \times \{d\} \bigr).
\]
The latter set coincides with $\mathbb{V} \cap ( \mathbb{A}\times \{d\} )$, so it is mapped into 
\[
\mathbb{V} \cap ( \mathbb{A}\times \{1\} ) = \T \cap   ( \mathbb{A}\times \{1\} )
\]
by the permutation $\sigma_{\binom{d}{2}} = \kappa_d$. Since $\mathbb{W}(\T)=\emptyset$, we conclude that
\[
\mathbb{W} \, \sigma_{\binom{d}{2}} \, \mathbb{W} \, \sigma_{\binom{d}{2}-1} \, \dots \, \sigma_2 \, \mathbb{W} = \emptyset.
\]
This proves the induction step and concludes the proof of the lemma.
\end{proof}

The above lemma can be rephrased as a vanishing property for the composition of suitable sequences of endomorphisms of $\mathscr{H}$:

\begin{lem}
\label{nilp}
Let $\mathscr{M}_0,\mathscr{M}_1,\dots, \mathscr{M}_{\binom{d}{2}}$ be linear endomorphisms of $\mathscr{H}$ such that
\begin{equation}
\label{hyp}
\supp \mathscr{M}_h \subset \mathbb{W}, \quad \mbox{for any } 0\leq h \leq \binom{d}{2},
\end{equation}
and let $\sigma_1,\sigma_2,\dots,\sigma_{\binom{d}{2}}$ be the permutations of $\mathbb{J}$ which are introduced in Lemma \ref{combi}. Then
\begin{equation}
\label{vanish}
\mathscr{M}_0 \, \mathscr{A}_{U_{\sigma_1}} \, \mathscr{M}_1\, \mathscr{A}_{U_{\sigma_2}}\, \mathscr{M}_2 \, \dots\, \mathscr{A}_{U_{\sigma_{\binom{d}{2}}}} \, \mathscr{M}_{\binom{d}{2}} = 0.
\end{equation}
\end{lem}

\begin{proof}
We can equivalently show that the support of the composition which appears on the left-hand side of (\ref{vanish}) is empty. By using (\ref{contro}),
(\ref{hyp}) implies that the support of this composition is contained in the following composition of relations
\[
\mathbb{W} \circ \supp \mathscr{A}_{U_{\sigma_{\binom{d}{2}}}} \circ \mathbb{W} \circ \cdots \circ \mathbb{W} \circ \supp \mathscr{A}_{U_{\sigma_1}} \circ \mathbb{W}.
\]
By (\ref{perm}), the above relation  is the set
\[
\mathbb{W} \, \sigma_{\binom{d}{2}} \, \mathbb{W} \, \dots \, \mathbb{W} \, \sigma_1 \, \mathbb{W},
\]
which is empty by Lemma \ref{combi}. 
\end{proof} 

\section{The conjugacy equation for 2-jets}
\label{2jetssec}

Consider the finite sequence $(\sigma_1,\dots,\sigma_{\binom{d}{2}})$ of permutations of $\mathbb{J}$ which is introduced in Lemma \ref{combi} and extend it to a $\binom{d}{2}$-periodic infinite sequence $(\sigma_h)_{h\in \N}$. Set 
\[
D:= 2^{d-1},
\]
and define the following two sequences of permutations of $\mathbb{J}$:
\[
\begin{split}
\theta_n &:= \left\{ \begin{array}{ll} \sigma_h & \qquad \mbox{if } n=D^h-1 \mbox{ for } h\in \N, \\ \mathrm{id} & \qquad \mbox{otherwise,} \end{array} \right.
\\
\tau_n &:= \left\{ \begin{array}{ll} \mathrm{id} & \mbox{if } n = 0, \\ \tau_{n-1} \circ \theta_{n-1} & \mbox{if } n\geq 1. \end{array} \right.
\end{split} 
\]
A family of $k$-jets of self-maps of $\C^d$ is said to be bounded if all their coefficients are uniformly bounded.
The aim of this section is to prove the following result:

\begin{prop}
\label{con2jets}
Let $(f_n)$ be a bounded sequence of 2-jets of invertible self-maps of $\C^d$ such that $f_n(0)=0$ and such that the sequence of linear endomorphisms $(Df_n(0))$ is $(\Lambda,M)$-pinched, where $\Lambda < 1 < M$ satisfy
\begin{equation}
\label{conve}
(\Lambda^2 M)^{D^{d(d-1)}} (\Lambda M)^{-\delta} < 1,
\end{equation}
where
\[
\delta:= \frac{D-1}{D^{d(d-1)}-1}.
\]
Then there exist sequences of 2-jets $(h_n)$, $(g_n)$ such that:
\begin{enumerate}
\item $h_n(0)=g_n(0)=0$ and $Dh_n(0)$ is unitary; 
\item $h_{n+1} \circ f_n = g_n \circ h_n$ as 2-jets, for every $n\in \N$.
\item $\tilde{g}_n := U_{\tau_n} \circ g_n \circ U_{\tau_n}^{-1}$ is an upper triangular polynomial automorphism of $\C^d$ of degree 2. 
\item for every $\Theta>\Lambda^2 M$ there is a number $C=C(\Theta)$ such that
\[
\|h_n\|\leq C \,\Theta^{(D^{3\binom{d}{2}}-1)n}, \qquad \|g_n\| \leq 
C \, \Theta^{(D^{3\binom{d}{2}}-1)n},
\]
for every $n\in \N$.
\end{enumerate}
\end{prop}

We set 
\[
\tilde{f}_n := U_{\tau_n} \circ f_n \circ U_{\tau_n}^{-1}.
\]
Up to a non-autonomous unitary conjugacy, we may assume that the linear automorphism 
\[
L_n:= D\tilde{f_n}(0) 
\]
is upper triangular for every $n\in \N$. Indeed, a non-autonomous conjugacy
\[
V_{n+1} \circ L_n = D\tilde{f}_n(0) \circ V_n
\]
between $(D\tilde{f_n}(0))$ and a sequence of upper triangular linear automorphisms $L_n$ can be defined recursively by setting $V_0:= I$ and by defining $V_{n+1}$ and $L_n$ to be the unitary and the upper triangular part in the $QR$-decomposition of $D\tilde{f}_n(0) \circ V_n$.

If we set
\[
\tilde{h}_n := U_{\tau_n} \circ h_n \circ U_{\tau_n}^{-1},
\]
the conjugacy equation in statement (i) can be rewritten as
\begin{equation}
\label{conju}
\tilde{h}_n = \tilde{g}_n^{-1} \circ U_{\theta_n}^{-1} \circ \tilde{h}_{n+1} \circ U_{\theta_n} \circ \tilde{f}_n.
\end{equation}
At the level of 1-jets, this equation is solved by choosing
\[
D \tilde{h}_n (0) := I, \qquad  D \tilde{g}_n(0) := L_n.
\]
If $\tilde{g}_n$ is an upper triangular polynomial automorphism of $\C^d$ of degree 2, so is its inverse, and we can set
\[
\tilde{g}_n^{-1} = L_n^{-1} + v_n,
\]
where $v_n\in \mathscr{T}=\mathscr{T}^2$ is to be found.
If we also set
\[
\tilde{h}_n = I + u_n,  
\]
with $u_n \in \mathscr{H}= \mathscr{H}^2$, the equation (\ref{conju}) at the level of 2-jets can be rewritten as
\begin{equation}
\label{conju2}
u_n = \mathscr{A}_{U_{\theta_n} L_n} u_{n+1} + v_n \circ L_n + L_n^{-1} \circ w_n,
\end{equation}
where $w_n$ denotes the 2-homogeneous part of $\tilde{f}_n$ (which is a given bounded sequence in $\mathscr{H}$). We recall that $\mathscr{Q}$ is the projector onto the space spanned the $z^{\alpha} e_i$ for $(\alpha,i)\in \T^c$, along $\mathscr{T}$. We shall find a solution $(u_n)$, $(v_n)$ of (\ref{conju2}) with
\[
u_n \in \mathscr{Q} \mathscr{H}.
\]
With such an Ansatz, the equation (\ref{conju2}) is equivalent to the system
\begin{eqnarray}
\label{conju3a}
u_n &=& \mathscr{Q} \left( \mathscr{A}_{U_{\theta_n} L_n} u_{n+1} + L_n^{-1} \circ w_n \right), \\
\label{conju3b}
v_n \circ L_n &=& (\mathscr{Q}-I) \left( \mathscr{A}_{U_{\theta_n} L_n} u_{n+1} + L_n^{-1} \circ w_n \right),
\end{eqnarray}
where we have used the fact that $v_n \circ L_n$ belongs to $\mathscr{T}=\ker \mathscr{Q}$. In order to prove Proposition \ref{con2jets}, it is enough to find a sequence $(u_n)$ in the image of $\mathscr{Q}$ which solves (\ref{conju3a}) and has the estimate
\[
\| u_n \| \leq C \,\Theta^{(D^{3\binom{d}{2}}-1)n}, \qquad \forall n\in \N,
\]
for every $\Theta>\Lambda^2 M$ and for some
for some $C=C(\Theta)$. Indeed, in this case $v_n$ can be derived immediately from the equation (\ref{conju3b}): It belongs to $\mathscr{T}$ because $\mathscr{T}$ is invariant with respect to the composition by the upper triangular linear mappings $L_n$ and it satisfies an analogous growth estimate. It is immediate to check that a solution of (\ref{conju3a}) can be defined explicitly by setting
\[
\begin{split}
u_n &:= \sum_{m\geq n} (\mathscr{Q} \mathscr{A}_{U_{\theta_n} L_n}) (\mathscr{Q} \mathscr{A}_{U_{\theta_{n+1}} L_{n+1}}) \dots ( \mathscr{Q} \mathscr{A}_{U_{\theta_{m-1}} L_{m-1}})  \mathscr{Q} L_m^{-1} w_m \\ &=
\sum_{m\geq n} (\mathscr{Q} \mathscr{A}_{U_{\theta_n} L_n} \mathscr{Q}) (\mathscr{Q} \mathscr{A}_{U_{\theta_{n+1}} L_{n+1}} \mathscr{Q}) \dots ( \mathscr{Q} \mathscr{A}_{U_{\theta_{m-1}} L_{m-1}} \mathscr{Q} ) L_m^{-1} w_m, 
\end{split} 
\]
provided that the above series converges in $\mathscr{H}$. Therefore,
Proposition \ref{con2jets} is implied by the following:

\begin{prop}
\label{operatore}
Let $(L_n)$ be a $(\Lambda,M)$-pinched sequence of upper triangular linear automorphisms of $\C^d$, where $\Lambda$ and $M$ satisfy (\ref{conve}). Then for every $n\in \N$ the series
\[
\mathscr{S}_n := \sum_{m\geq n} (\mathscr{Q} \mathscr{A}_{U_{\theta_n} L_n} \mathscr{Q}) (\mathscr{Q} \mathscr{A}_{U_{\theta_{n+1}} L_{n+1}} \mathscr{Q}) \dots ( \mathscr{Q} \mathscr{A}_{U_{\theta_{m-1}} L_{m-1}} \mathscr{Q} )
\]
converges absolutely in $\mathrm{L}(\mathscr{H})$. Moreover, for every $\Theta>\Lambda^2 M$ there exists $C=C(\Theta)$ such that
\[
\|\mathscr{S}_n\| \leq C\, \Theta^{(D^{3\binom{d}{2}}-1)n},
\]
for every $n\in \N$.
\end{prop}

The remaining part of this section is devoted to the proof Proposition \ref{operatore}.
If $m>n$ are natural numbers, we set
\[
S(m,n) := \bigl\| (\mathscr{Q} \mathscr{A}_{U_{\theta_n} L_n} \mathscr{Q}) (\mathscr{Q} \mathscr{A}_{U_{\theta_{n+1}} L_{n+1}} \mathscr{Q}) \dots ( \mathscr{Q} \mathscr{A}_{U_{\theta_{m-1}} L_{m-1}} \mathscr{Q} ) \bigr\|,
\]
and $S(n,n):=1$. Since the norm of a composition is not larger than the products of the norms of the factors, we have
\begin{equation}
\label{sub}
S(\ell,n) \leq S(\ell,m) S(m,n), \qquad \mbox{for every } 0\leq n \leq m \leq \ell.
\end{equation}

\begin{lem}
\label{treniS}
Assume that the sequence $(L_n)$ of upper triangular automorphisms of $\C^d$ is $(\Lambda,M)$-pinched.
\begin{enumerate}
\item For every $\Theta>\Lambda^2 M$ there is a positive number $C_1=C_1(\Theta)$ such that
\[
S(m,n) \leq C_1 \, \Theta^{m-n}, 
\]
for every $m\geq n \geq 0$.
\item For every $\Theta>\Lambda^2 M$ there is a positive number $C_2=C_2(\Theta)$ such that
\[
S\left( D^{(p+2q)\binom{d}{2}}, D^{p\binom{d}{2}} \right) \leq C_2 \, \bigl( \Theta \Lambda^{-\delta} M^{- \delta} \bigr)^ {D^{p \binom{d}{2}} ( D^{2q \binom{d}{2} } - 1 )},
\]
for every $p,q\in \N$, where $\delta$ is the number introduced in Lemma \ref{operatore}.
\end{enumerate}
\end{lem}

\begin{proof}
(i) Let $n\leq m$ be natural numbers and let $h\leq k$ be the natural numbers such that
\begin{equation}
\label{dove}
D^{h-1} -1 < n \leq D^h -1 , \qquad D^{k-1} - 1 < m \leq D^k - 1.
\end{equation}
Since $\mathscr{Q}$ commutes with the operators $\mathscr{A}_{L_j}$, the composition whose norm defines $S(m,n)$ can be rewritten as
\begin{eqnarray*}
& (\mathscr{Q} \mathscr{A}_{U_{\theta_n} L_n} \mathscr{Q}) (\mathscr{Q} \mathscr{A}_{U_{\theta_{n+1}} L_{n+1}} \mathscr{Q}) \dots ( \mathscr{Q} \mathscr{A}_{U_{\theta_{m-1}} L_{m-1}} \mathscr{Q} ) & \\ & 
= ( \mathscr{Q} \mathscr{A}_{L_{D^h,n}} \mathscr{Q}) \mathscr{A}_{U_{\sigma_h}}
( \mathscr{Q} \mathscr{A}_{L_{D^{h+1},D^h}} \mathscr{Q}) \mathscr{A}_{U_{\sigma_{h+1}}} \dots (\mathscr{Q} \mathscr{A}_{L_{D^{k-1},D^{k-2}}} \mathscr{Q}) \mathscr{A}_{U_{\sigma_{k-1}}} (\mathscr{Q} \mathscr{A}_{L_{m,D^{k-1}}} \mathscr{Q}), &
\end{eqnarray*}
where we have also used the definition of $(\theta_j)$. If $C$ is such that (\ref{est}) holds, then (\ref{norma}) and the above expression imply the estimate
\[
S(m,n) \leq c \,(\lambda^2 M)^{D^h-n} c \,(\Lambda^2 M)^{D^{h+1}-D_h} \dots  c \,(\Lambda^ M)^{m-D^k} = c^{k-h+1} \, (\Lambda^2 M)^{m-n},
\]
where $c = C^3 \|\mathscr{Q}\|^2$. By (\ref{dove}), we have
\[
k-h+1 \leq 2 + \frac{1}{\log D} \log \frac{m+1}{n+1},
\]
hence 
\[
S(m,n) \leq c^2 \left( \frac{m+1}{n+1} \right)^{\frac{\log c_1}{\log D}} (\Lambda^2 M)^{m-n}.
\]
Claim (i) follows.

\medskip

\noindent (ii) By using the periodicity of $(\sigma_h)$, the composition whose norm is $S( D^{(p+2q)\binom{d}{2}}, D^{p\binom{d}{2}})$ can be rewritten as
\begin{equation}
\label{lunga}
\begin{split}
\Bigl( \mathscr{Q} \mathscr{A}_{L_{D^{p\binom{d}{2}+1},D^{p\binom{d}{2}}}} \mathscr{Q} \Bigr) \mathscr{A}_{U_{\sigma_1}} \Bigl( \mathscr{Q} \mathscr{A}_{L_{D^{p\binom{d}{2}+2},D^{p\binom{d}{2}+1}}} \mathscr{Q} \Bigr) \mathscr{A}_{U_{\sigma_2}} \dots \\ \dots\mathscr{A}_{\sigma_{2q\binom{d}{2}-1}}  \Bigl( \mathscr{Q} \mathscr{A}_{L_{D^{(p+2q)\binom{d}{2}+2},D^{(p+2q)\binom{d}{2}-1}}} \mathscr{Q} \Bigr) \mathscr{A}_{U_{\sigma_{2q{\binom{d}{2}}}}}.
\end{split}
\end{equation}
By Lemma \ref{decomp}, we have the decompositions
\[
\mathscr{Q} \mathscr{A}_{L_{D^{p \binom{d}{2} + j + 1}, D^{p \binom{d}{2} + j}}} \mathscr{Q} = \mathscr{M}_j^0 + \mathscr{M}_j^1, \qquad \forall j=0,1,\dots, 2q \binom{d}{2} -1,
\]
where
\begin{equation}
\label{deco}
\supp \mathscr{M}_j^1 \subset \mathbb{W}, \qquad \left\{ \begin{array}{ccl}  \bigl\|\mathscr{M}_j^0\bigr\|  &\leq& C _0\, \theta^{\ell(j)} \, \Lambda^{\ell(j)}, \\   \bigl\|\mathscr{M}_j^1\bigr\|  &\leq& C_1\,  (\Lambda^2 M)^{\ell(j)}, \end{array} \right. \qquad \ell(j) := D^{p\binom{d}{2}+j} (D-1),
\end{equation}
for every $\theta>1$, for some $C_0=C_0(\theta)$ and $C_1$. Notice that
\begin{equation}
\label{lasomma}
\sum_{j=0}^{2q\binom{d}{2}-1} \ell(j) = D^{p\binom{d}{2}}(D^{2q\binom{d}{2}}-1).
\end{equation}
Therefore, the composition (\ref{lunga}) can be written as
\begin{equation}
\label{sommona}
\sum_{\epsilon \in \mathbbm{2}^{2q\binom{d}{2}}} \mathscr{M}_0^{\epsilon(0)} \mathscr{A}_{U_{\sigma_1}}  \mathscr{M}_1^{\epsilon(1)} \mathscr{A}_{U_{\sigma_2}} \dots  \mathscr{A}_{U_{\sigma_{2q\binom{d}{2}-1}}} \mathscr{M}_{2q \binom{d}{2} -1}^{\epsilon\bigl(2q \binom{d}{2} -1\bigr)} \mathscr{A}_{U_{\sigma_{2q\binom{d}{2}}}}.
\end{equation}
Since the support of each $\mathscr{M}_j^1$ is contained in $\mathbb{W}$, Lemma \ref{nilp} implies that if $\epsilon \in \mathbbm{2}^{2q\binom{d}{2}}$ is such that
\[
\epsilon(j)=1 \qquad \mbox{for} \quad  p + r \binom{d}{2} \leq  j \leq p + (r+1) \binom{d}{2},
\]
for some integer $0\leq r < q$, then the corresponding term in the sum (\ref{sommona}) vanishes. We denote by $E\subset \mathbbm{2}^{2q\binom{d}{2}}$ the complementary set of $\epsilon$'s, that is the set of all $\epsilon \in \mathbbm{2}^{2q\binom{d}{2}}$ such that for every $0\leq r < q$ there is an index $j$ between $p + r \binom{d}{2}$ and $p + (r+1) \binom{d}{2}$ such that $\epsilon(j)=0$. Then $S( D^{(p+2q)\binom{d}{2}}, D^{p\binom{d}{2}})$, that is the norm of the sum (\ref{sommona}), can be estimated using (\ref{deco}) by
\begin{equation}
\label{gg1}
\begin{split}
& S\left(  D^{(p+2q)\binom{d}{2}}, D^{p\binom{d}{2}}\right)  \leq  2^{2q\binom{d}{2}} \max_{\epsilon\in E} \prod_{j=0}^{2q \binom{d}{2}-1} C_2\, (\Lambda^2 M)^{\epsilon(j) \ell(j)} \theta^{(1-\epsilon(j)) \ell(j)} \Lambda^{(1-\epsilon(j)) \ell(j)} \\
& \leq 2^{2q\binom{d}{2}} \max_{\epsilon\in E} \prod_{j=0}^{2q \binom{d}{2}-1} C_2\, (\Lambda^2 M)^{\epsilon(j) \ell(j)} \theta^{\ell(j)} \Lambda^{(1-\epsilon(j)) \ell(j)} \\
& =  (2C_2)^{2q\binom{d}{2}} \theta^{D^{p\binom{d}{2}}(D^{2q\binom{d}{2}}-1)} \max_{\epsilon\in E} \left( (\Lambda^2 M)^{\sum_{j=0}^{2q \binom{d}{2}-1} \epsilon(j)\ell(j)} \Lambda^{\sum_{j=0}^{2q \binom{d}{2}-1} (1-\epsilon(j)) \ell(j)}\right),
\end{split}
\end{equation}
where $C_2=C_2(\theta) = \max\{C_0(\theta),C_1\}$ and we have used (\ref{lasomma}).
By multiplying and dividing by
\[
(\Lambda M)^{\sum_{j=0}^{2q \binom{d}{2}-1} (1-\epsilon(j)) \ell(j)},
\]
and by using (\ref{lasomma}), we obtain
\begin{equation}
\label{gg2}
\begin{split}
\max_{\epsilon\in E} & \left( (\Lambda^2 M)^{\sum_{j=0}^{2q \binom{d}{2}-1} \epsilon(j)\ell(j)} \Lambda^{\sum_{j=0}^{2q \binom{d}{2}-1} (1-\epsilon(j)) \ell(j)}\right) \\
&=  (\Lambda^2 M)^{D^{p\binom{d}{2}}(D^{2q\binom{d}{2}}-1)} \max_{\epsilon \in E}\, (\Lambda M)^{- \sum_{j=0}^{2q \binom{d}{2}-1} (1-\epsilon(j)) \ell(j)} \\
&=  (\Lambda^2 M)^{D^{p\binom{d}{2}}(D^{2q\binom{d}{2}}-1)} (\Lambda M)^{- \min_{\epsilon \in E} \sum_{j=0}^{2q \binom{d}{2}-1} (1-\epsilon(j)) \ell(j) },
\end{split}
\end{equation}
where we have also used the fact that $\Lambda M$ is greater than 1.
The last minimum is achieved for $\epsilon$ such that the sum of the terms $\ell(j)$ over all $j$ for which $\epsilon(j)$ is zero is minimal among all $\epsilon$ in $E$. By the definition of $E$ and by the fact that the function $j\mapsto \ell(j)$ is increasing, we deduce that this minimum is achieved at
\[
\epsilon(j) = \left\{ \begin{array}{ll} 0 & \mbox{if } j=(p+2r) \binom{d}{2} \mbox{ with } 0\leq r < q \\ 1 & \mbox{otherwise,} \end{array} \right.
\]
and we find
\begin{equation}
\label{gg3}
\begin{split}
\min_{\epsilon \in E} \sum_{j=0}^{2q \binom{d}{2}-1} (1-\epsilon(j)) \ell(j)  &= \sum_{r=0}^{q-1} \ell \bigl( 2r {\scriptstyle \binom{d}{2}} \bigr) = \sum_{r=0}^{q-1} D^{(p+2r)\binom{d}{2}} (D-1) \\ &= \frac{D-1}{D^{2\binom{d}{2}}-1} D^{p \binom{d}{2}} \left( D^{2q \binom{d}{2} } - 1 \right) = \delta \, D^{p \binom{d}{2}} \left( D^{2q \binom{d}{2} } - 1 \right).
\end{split}
\end{equation}
Putting together (\ref{gg1}), (\ref{gg2}) and (\ref{gg3}), we obtain
\[
\begin{split}
S\left( D^{(p+2q)\binom{d}{2}}, D^{p\binom{d}{2}}\right) &\leq (2C_2)^{2q\binom{d}{2}} \theta^{D^{p\binom{d}{2}}(D^{2q\binom{d}{2}}-1)}  (\Lambda^2 M)^{D^{p\binom{d}{2}}(D^{2q\binom{d}{2}}-1)} (\Lambda M)^{- \delta  D^{p \binom{d}{2}} ( D^{2q \binom{d}{2} } - 1 )} \\
& =  (2C_2)^{2q\binom{d}{2}} \theta^{D^{p\binom{d}{2}}(D^{2q\binom{d}{2}}-1)}  \bigl( \Lambda^{2-\delta} M^{1- \delta} \bigr)^ {D^{p \binom{d}{2}} ( D^{2q \binom{d}{2} } - 1 )}.
\end{split}
\] 
The desired estimate follows, because $\theta>1$ is arbitrary and from the fact that 
\[
(2C_2)^{2q\binom{d}{2}} = O\Bigl(\sigma^{D^{2q\binom{d}{2}}} \Bigr) \qquad \mbox{for} \quad q\rightarrow \infty,
\]
for any $\sigma>0$.
\end{proof}

\begin{proof}[Proof of Proposition \ref{operatore}] Fix some $\Theta>\Lambda^2 M$. By (\ref{conve}), up to the choice of a smaller $\Theta$ we can assume that
\begin{equation}
\label{conve2}
\eta = \eta(\Theta) := \Theta^{D^{2 \binom{d}{2}}} (\Lambda M)^{-\delta} < 1.
\end{equation}
We must prove that there exists a positive number $C=C(\Theta)$ such that
\begin{equation}
\label{lastima}
\sum_{m\geq n} S(m,n) \leq C\, \Theta^{(D^{3\binom{d}{2}}-1)n},
\end{equation}
for every $n\in \N$. Let $p$ be a natural number. By (\ref{sub}) we have 
\begin{equation}
\label{ppp1}
\begin{split}
 \sum_{m\geq D^{p\binom{d}{2}} } & S \bigl( m, D^{p\binom{d}{2}} \bigr) = \sum_{q\geq 0} \sum_{D^{(p+2q)\binom{d}{2}} \leq m < D^{(p+2(q+1))\binom{d}{2}} }  \hspace{-1cm} S \bigl( m, D^{p\binom{d}{2}} \bigr) \\
& \leq \sum_{q\geq 0} S \bigl( D^{(p+2q)\binom{d}{2}}, D^{p\binom{d}{2}} \bigr) \hspace{-1cm}
\sum_{D^{(p+2q)\binom{d}{2}} \leq m < D^{(p+2(q+1))\binom{d}{2}} } \hspace{-1cm} S \bigl( m, D^{(p+2q)\binom{d}{2}} \bigr).
\end{split} \end{equation}
The inner sum can be estimated using Lemma \ref{treniS} (i) as
\begin{equation}
\label{ppp2}
\begin{split}
\sum_{D^{(p+2q)\binom{d}{2}} \leq m < D^{(p+2(q+1))\binom{d}{2}} }  \hspace{-1cm} S \bigl( m, D^{(p+2q)\binom{d}{2}} \bigr) \leq C_1 \hspace{-1cm} \sum_{D^{(p+2q)\binom{d}{2}} \leq m < D^{(p+2(q+1))\binom{d}{2}} } \hspace{-1cm}  \Theta^{m-D^{(p+2q)\binom{d}{2}}} \\ = C_1 \,\frac{\Theta^{D^{(p+2q)\binom{d}{2}} ( D^{2\binom{d}{2}} - 1 )} - 1}{\Theta - 1} \leq C_3\,  \Theta^{D^{(p+2q)\binom{d}{2}} ( D^{2\binom{d}{2}} - 1 )},
\end{split} \end{equation}
where $C_3 = C_3(\Theta) := C_1(\Theta) /(\Theta - 1)$. By (\ref{ppp1}), (\ref{ppp2}) and Lemma \ref{treniS} (ii), we obtain the estimate
\begin{equation}
\label{ppp3}
\begin{split}
\sum_{m\geq D^{p\binom{d}{2}} }  S \bigl( m, D^{p\binom{d}{2}} \bigr) &\leq  \sum_{q\geq 0} C_2 \,( \Theta \Lambda^{-\delta} M^{-\delta} )^{D^{p\binom{d}{2}}( D^{2q\binom{d}{2}} - 1)} C_3 \, \Theta^{D^{(p+2q)\binom{d}{2}} ( D^{2\binom{d}{2}} - 1 )} \\
&= C_2 C_3 \, (\Theta \Lambda^{-\delta} M^{-\delta})^{-D^{p\binom{d}{2}}} \sum_{q\geq 0} \Bigl( \Theta^{D^{2\binom{d}{2}}} (\Lambda M)^{-\delta} \Bigr)^{D^{(p+2q)\binom{d}{2}}} \\
&= C_2 C_3 \, (\Theta \Lambda^{-\delta} M^{-\delta})^{-D^{p\binom{d}{2}}} \sum_{q\geq 0} \eta^{D^{(p+2q)\binom{d}{2}}}.
\end{split}
\end{equation}
By (\ref{conve2}), the positive number $\eta$ is smaller than 1, so the above series converges. More precisely, 
\begin{equation}
\label{ppp4}
\sum_{q\geq 0} \eta^{D^{(p+2q)\binom{d}{2}}} \leq \sum_{j\geq 1} \eta^{D^{p\binom{d}{2}}j} = \frac{\eta^{D^{p\binom{d}{2}}}}{1-\eta^{D^{p\binom{d}{2}}}} \leq \frac{1}{1-\eta} \, \eta^{D^{p\binom{d}{2}}}.
\end{equation}
By (\ref{ppp3}) and (\ref{ppp4}), we find the upper bound
\begin{equation}
\label{ppp5}
\sum_{m\geq D^{p\binom{d}{2}} }  S \bigl( m, D^{p\binom{d}{2}} \bigr) \leq C_4 (\Theta \Lambda^{-\delta} M^{-\delta})^{-D^{p\binom{d}{2}}} \eta^{D^{p\binom{d}{2}}} = C_4 \Bigl( \Theta^{D^{2\binom{d}{2}} - 1} \Bigr)^{D^{p\binom{d}{2}}}.
\end{equation}
Now we fix an arbitrary natural number $n$ and we let $p$ be the natural number such that
\begin{equation}
\label{dove2}
D^{(p-1)\binom{d}{2}} < n+ 1 \leq D^{p\binom{d}{2}}.
\end{equation}
By (\ref{sub}) we have
\begin{equation}
\label{ppp6}
\begin{split}
\sum_{m\geq n} S(m,n) &= \sum_{n\leq m < D^{p\binom{d}{2}}} S(n,m) + \sum_{m\geq D^{p\binom{d}{2}}} S(n,m) \\ &\leq \sum_{n\leq m < D^{p\binom{d}{2}}} S(n,m) + S\bigl(D^{p\binom{d}{2}},n\bigr) \sum_{m\geq D^{p\binom{d}{2}}} S\bigl(D^{p\binom{d}{2}},m\bigr).
\end{split} \end{equation}
By Lemma \ref{treniS} (i), the first sum in the above line has the upper bound
\[
\sum_{n\leq m < D^{p\binom{d}{2}}} S(n,m) \leq C_1 \hspace{-.5cm} \sum_{n\leq m < D^{p\binom{d}{2}}} \Theta^{m-n} \leq \frac{C_1}{\Theta - 1}\, \Theta^{D^{p\binom{d}{2}}-n},
\]
Also the term $S(D^{p\binom{d}{2}},n)$ can be estimated by using Lemma \ref{treniS} (i), while the last sum in (\ref{ppp6}) has the upper bound (\ref{ppp5}), and  we obtain
\begin{equation}
\label{ppp7}
\sum_{m\geq n} S(m,n) \leq \frac{C_1}{\Theta - 1}\, \Theta^{D^{p\binom{d}{2}}-n} + C_1 \,  \Theta^{D^{p\binom{d}{2}}-n} C_4\, \Bigl( \Theta^{D^{2\binom{d}{2}} - 1} \Bigr)^{D^{p\binom{d}{2}}}.
\end{equation}
By (\ref{dove2}),
\[
D^{p\binom{d}{2}} < D^{\binom{d}{2}} (n+1),
\]
so (\ref{ppp7}) implies
\begin{equation}
\label{ppp8}
\begin{split}
\sum_{m\geq n} S(m,n) &\leq C_5 \, \Theta^{(D^{\binom{d}{2}}-1)n} + C_6\,   
\Theta^{(D^{\binom{d}{2}}-1)n}
\Bigl( \Theta^{D^{2\binom{d}{2}} - 1} \Bigr)^{D^{\binom{d}{2}} n} \\
&= C_5 \, \Theta^{(D^{\binom{d}{2}}-1)n} + C_6\,   
\Theta^{(D^{3\binom{d}{2}}-1)n},
\end{split}
\end{equation}
for suitable numbers $C_5,C_6$. The estimate (\ref{ppp8}) implies that there exists $C$ such that
\[
\sum_{m\geq n} S(m,n) \leq C \, \Theta^{(D^{3\binom{d}{2}}-1)n},
\]
proving (\ref{lastima}).
\end{proof}

\section{The basin of attraction of a class of sequences}

We recall that if $f=(f_n)$ is a sequence of automorphisms of $\C^d$ which fix the origin, then the basin of attraction of the origin with respect to $f$ is the set
\[
\set{z\in \C^d}{f_{n,0}(z) \rightarrow 0 \mbox{ for } n\rightarrow \infty}.
\]
Here we are using the notation
\[
f_{n,m} := f_{n-1} \circ \dots \circ f_m, \qquad \forall n > m \geq 0, \qquad f_{n,n} = \mathrm{id},
\]
which we have used so far for linear mappings.
The aim of this section is to exhibit a useful class of sequences of automorphisms of $\C^d$ which have the whole $\C^d$ as basin of attraction of the origin.

Let $f_n \colon  \C^d \rightarrow \C^d$ be defined as
\begin{equation}
\label{laforma}
f_n (z) := L_n z + p_n(z)\;,
\end{equation}
where:
\renewcommand{\theenumi}{\alph{enumi}}
\renewcommand{\labelenumi}{(\theenumi)}
\begin{enumerate}
\item $(L_n)$ is a sequence of upper triangular linear automorphisms of $\C^d$
 such that $\|L_{n,m}\|\leq C\, \Lambda^{n-m}$ for every $n\geq m \geq 0$, where $C>0$ and $0<\Lambda<1$;
\item $(p_n)$ is a bounded sequence of polynomial maps (that is, the degree of $p_n$ and its coefficients are uniformly bounded) of the form
\[
p_n (z_1,\dots, z_d ) = \bigl(p_n^1(z_2,\dots,z_d), p_n^2(z_3,\dots,z_d),
\dots, p_n^{d-1} (z_d) , 0\bigr),
\]
such that  $p_n(0)=0$, $Dp_n(0)=0$.
\end{enumerate}
It is well-known and easy to show that each $f_n$ is an automorphism of $\C^d$. We refer to maps of this form as to ``special triangular automorphisms''.
Since the degree of $p_n$ is bounded, there exist positive integers $k_1,k_2, \dots,k_{d-1},k_d=1$ such that
\begin{equation}
\label{gradi}
\deg p_n^j \bigl(z_{j+1}^{k_{j+1}},z_{j+2}^{k_{j+2}}, \dots, z_d^{k_d} \bigr) \leq k_j\;, \quad \forall j=1,\dots,d-1\;.
\end{equation}
It is easy to show that the composition of two polynomial maps of the form (\ref{laforma}) which satisfy (\ref{gradi}) has the same properties (see \cite[Lemma 6.2]{aam11}). Hence, the number 
\[
K:=  \max\{k_1,\dots,k_d\}
\] 
may be called the ``stable degree'' of the sequence of maps $(f_n)$.  When all the polynomial maps $p_n$ have degree at most 2, then we can take $k_j = 2^{d-j}$ in (\ref{gradi}) and hence the stable degree is $2^{d-1}$.
The following result is proved in \cite[Lemma 6.4]{aam11}:

\begin{lem}
\label{triang}
Let $(f_n)$ be a sequence of special triangular automorphisms of $\C^d$ of the form (\ref{laforma}) which satisfies (a), (b), (\ref{gradi}) and has stable degree $K$. Then there exists a number $C_0$ such that
\[
|f_{n,0} (z)| \leq C_0 \, \Lambda^n ( |z| + |z|^K)\;, \qquad \forall z\in \C^d\;,
\]
for every $n\in \N$. 
\end{lem}

In particular, the basin of attraction of a sequence of automorphisms $(f_n)$ which satisfies the assumptions of the above lemma is the whole $\C^d$. The next result shows that this fact remains true if some of the $f_n$'s, but not too many, are  composed by a linear automorphism, which may destroy the triangular structure:

\begin{thm}
\label{treni}
Let $(f_n)$ be a sequence of special triangular automorphisms of $\C^d$ of the form (\ref{laforma}) which satisfies (a), (b), (\ref{gradi}) and has stable degree $K$. Let $(T_h)$ be a bounded sequence of linear automorphisms of $\C^d$, let $(m_h)$ be a strictly increasing sequence of natural numbers, and set
\[
g_n := \left\{ \begin{array}{ll} f_n & \mbox{if } n\neq m_h, \; \forall h\in \N, \\ f_{m_h} \circ T_h & \mbox{if } n=m_h. \end{array} \right.
\]
If $K$ and the sequence $(m_h)$ satisfy
\begin{equation}\label{lunghi}
\lim_{h\rightarrow \infty} (m_{h+1} - m_h) = +\infty, \qquad \sum_{h\in \N} K^{-h} m_h = +\infty,
\end{equation}
then the basin of attraction of the sequence $(g_n)$ is the whole $\C^d$.
\end{thm}

The proof makes use of the following lemma:

\begin{lem}\
\label{tozz0}
Let $0<\Lambda<1$, $C>0$, $K>1$, and let $(s_h)$ be a diverging sequence of positive numbers. Then the sequence $(r_h)$ defined by 
\begin{equation}\label{raggetti}
\left\{ \begin{array}{ll} r_{h+1} = C \,\Lambda^{s_h} (r_h + r_h^K) & \forall h\in \N, \\ r_0 = r, & \end{array} \right.
\end{equation}
is infinitesimal for every initial value $r\geq 0$ if and only if 
\begin{equation}
\label{diverge}
\sum_{h\in \N} K^{-h} s_h = +\infty.
\end{equation}
\end{lem}

\begin{proof}
Let us show that condition (\ref{diverge}) is sufficient. Since $(s_h)$ diverges and $\Lambda<1$, up to a shift we may assume that
\[ 
C \, \Lambda^{s_h} \leq \frac{1}{4} , \quad \forall h\in \N.
\]
If $r_h\leq 1$ then
\[
r_{h+1} \leq \frac{1}{4} (r_h + r_h^K) \leq \frac{1}{4} ( r_h + r_h) = \frac{1}{2} r_h.
\]
The above fact implies that it is enough to show that for any initial value $r$ there is a $h\in \N$ such that $r_h \leq 1$. If we set $\rho_h = \log r_h$, we have
\[
\rho_{h+1} = \log C + s_h \log \Lambda + \rho_h + \log (1+ r_h^{K-1}).
\]
Since $\log (1+t)\leq 1 + \log t$ for every $t\geq 1$, if we assume that $r_h\geq 1$ we deduce that
\[
\rho_{h+1} \leq \log C + s_h \log \Lambda + \rho_h + 1 + (K-1) \log r_h 
 =K\, \rho_h + 1 + \log C + s_h \log \Lambda.
\]
The above inequality implies that, if $r_j \geq 1$ for every $j=0,\dots,h-1$ then
\[
\rho_h \leq K^h \rho_0 + \sum_{j=0}^{h-1} K^{h-j-1} ( 1 + \log C + s_j \log \Lambda) 
\leq K^h \left( \rho_0 + \frac{|1 + \log C|}{K-1} + \frac{\log \Lambda}{K} \sum_{j=0}^{h-1} K^{-j}  s_j \right),
\]
where we have estimated the finite sum of the $K^{-j}$'s by the sum of the corresponding series, using the hypothesis $K>1$. Since $\log \Lambda<0$, (\ref{diverge}) implies that  the latter quantity tends to $-\infty$ for $h\rightarrow +\infty$. This proves that there is a number $h\in \N$ such that $\rho_h \leq 0$, and hence $r_h \leq 1$, concluding the first part of the proof.

There remains to show that condition (\ref{diverge}) is necessary. From the inequality
\[
r_{h+1} \geq C \, \Lambda^{s_h} r_h^K, \quad \forall h\in \N,
\]
we deduce
\[
r_h \geq C^{\sum_{j=0}^{h-1} K^j} \Lambda^{\sum_{j=0}^{h-1} K^{h-j-1} s_j} r_0^{K^h} = C^{K^h/(K-1)} \Lambda^{K^{h-1} \sum_{j=0}^{h-1} K^{-j} s_j} r_0^{K^h}. 
\]
If the series $\sum_j K^{-j} s_j$ converges, the latter quantity tends to $+\infty$ when $r_0$ is large enough. The fact that $(r_h)$ is infinitesimal for every value of $r_0$ then implies (\ref{diverge}).
\end{proof} 

\begin{proof}[Proof of Theorem \ref{treni}] 
Notice that, by (\ref{lunghi}), the sequence $s_h := m_{h+1} - m_h$ diverges and that, summing by parts, 
\[
\sum_{h\in \N} K^{-h} s_h \geq (K-1)\sum_{h\in \N} K^{-h} m_h -K m_0=+\infty.
\]
Let $z\in\C^d$. By Lemma \ref{triang}, there holds
\[
|T_{h+1}f_{m_{h+1},m_h}(z)|\leq C\Lambda^{s_h}(|z|+|z|^K), \qquad \forall h\in \N,
\]
where $C:=C_0 \sup_h\Vert T_h\Vert$. Therefore the sequence
\[
z_h:=T_{h}f_{m_h ,m_{h-1} } \circ T_{h-1}f_{m_{h-1},m_{h-2}}\circ\dots\circ T_{1}f_{m_1,m_{0}}( T_0 z) 
\]
satisfies $|z_h|\le r_h$, where the sequence $ r_h$ is defined by (\ref{raggetti}) with $r:=|T_0z|$ and is infinitesimal by Lemma \ref{tozz0}. For any $n\in\N$ with $m_h<n\leq m_{h+1}$, Lemma \ref{triang} implies
\[
|g_{n,0}(z)|=|f_{n,m_h}(z_h)|\le C_0 \Lambda^{n-m_h} (r_h+r_h^K)\to0\quad \mbox{as }n\to \infty,
\]
so $z$ belongs to the basin of attraction of the origin with respect to $(f_n)$, which is therefore the whole $\C^d$.
\end{proof} 

\section{The abstract basin of attraction}
\label{abasec}

The aim of this section is to recall some definitions and results from \cite[Section 5]{aam11}. 
Let $\mathscr{G}$ be the category whose objects are the sequences
\[
f = (f_n : U_n \rightarrow U_{n+1})_{n\in \N}
\]
of injective holomorphic  maps between $d$-dimensional complex manifolds and whose morphisms $h\colon f\rightarrow g$ are sequences of injective holomorphic maps 
\[
h = (h_n \colon   U_n \rightarrow V_n)_{n\in \N}\; , \quad \mbox{with } U_n = \dom f_n, \; V_n = \dom g_n ,
\]
such that for every $n\in \N$ the diagram
\begin{equation*}
\begin{CD}
U_n @>{f_n}>> U_{n+1} \\ @V{h_n}VV @VV{h_{n+1}}V \\ V_n @>{g_n}>> V_{n+1}
\end{CD}
\end{equation*}
commutes. In other words, $\mathscr{G}$ is the category of functors $\mathrm{Fun} (\N,\mathscr{M})$, where $\mathscr{M}$ is the category of $d$-dimensional complex manifolds and injective holomorphic maps. 

We denote by $W$ the inductive limit functor 
\[
\mathop{\mathrm{Lim}}_{\longrightarrow} \colon  \mathscr{G} \rightarrow \mathscr{M}.
\]
That is, $Wf$ is the topological inductive limit of the sequence of maps $(f_n)$ with the induced holomorphic structure: Constructively, $Wf$ is the quotient of the set
\[
\left\{ (z_n)_{n\geq m} \in \prod_{n\geq m} U_n \; \Big| \; m\in \N\;, \; z_{n+1} = f_n(z_n) \; \forall n\geq m\right\} 
\]  
by identifying $z$ and $z'$ if $z_n=z'_n$ for $n$ large enough. 
The holomorphic structure is induced by the open inclusions
\[
f_{\infty,m} \colon  U_m \hookrightarrow Wf\;, \quad z \mapsto [(f_{n,m}(z))_{n\geq m}].
\] 
With the above representation, if $h\colon f\rightarrow g$ is a morphism in $\mathscr{G}$, $Wh$ is the injective holomorphic map 
\[
Wh ([(z_n)_{n\geq m}]) = [(h_n(z_n))_{n\geq m}].
\]
The following result, whose proof is immediate, turns our to be useful in order to identify $Wf$:

\begin{lem}
\label{biholo}
Let $h\colon f\rightarrow g$ be a morphism in $\mathscr{G}$. Then $Wh\colon  Wf \rightarrow Wg$ is surjective (hence a biholomorphism) if and only if for every $m\in \N$ and every $z\in V_m$ there exists $n\geq m$ such that $g_{n,m}(z)\in h_n(U_n)$.
\end{lem} 

Let $B=B_1$ be the open unit ball about the origin in $\C^d$ and 
consider a sequence of injective holomorphic maps $f_n \colon  B \rightarrow B$ such that
\begin{equation}
\label{contra}
|f_n(z)| \leq \Lambda |z| , \qquad \forall z\in B, \; \forall n\in \N, 
\end{equation}
for some $\Lambda<1$. In this case, each map $f_n$ fixes the origin and the manifold $Wf$ may be considered as the abstract basin of attraction of the origin with respect to the sequence $f=(f_n)$. In fact, if in addiction the maps $f_n$ are restrictions of global automorphisms $g_n$ of $\C^d$, then $g_{\infty,n}\colon \C^d\to Wg$ are biholomorphisms and, in particular, $Wg$ can be identified with $\C^d$; through this identification, the induced holomorphic inclusion $Wf \hookrightarrow Wg\cong \C^d$ is the inclusion in $\C^d$ of the open set
\[
 \set{z\in \C^d}{g_{n,0}(z)\rightarrow 0 \mbox{ for } n \rightarrow \infty} ,
\]
which is the basin of attraction of the origin with respect to $g$, 
and the maps $f_{\infty,n}\colon B \rightarrow Wf$ coincide with $g_{\infty,n}|_B$. Notice also that an immediate application of Lemma \ref{biholo} implies that if $f$ satisfies (\ref{contra}), then for every $r<1$ the manifold $Wf$ is biholomorphic to the abstract basin of attraction of the restriction
\[
f_n|_{B_r} \colon  B_r \rightarrow B_r, \qquad n\in \N .
\] 
By a bounded sequence of holomorphic germs we mean a sequence of holomorphic maps
\[
h_n \colon  B_r \rightarrow \C^d\; , \qquad n\in \N\; ,
\]
defined on the same ball of radius $r$ about the origin and such that $h_n(B_r)$ is uniformly bounded. 
Under boundedness assumptions, the abstract basin of attraction is invariant with respect to non-autonomous conjugacies, as shown by the following result (see \cite[Lemma 5.2]{aam11}):

\begin{lem}
\label{conimpW}
Let $f=(f_n\colon B \rightarrow B)_{n\in \N}$ and  $g=(g_n\colon B \rightarrow B)_{n\in \N}$ be objects in $\mathscr{G}$ such that
\[
|f_n(z)|\leq \Lambda |z|\; , \quad |g_n(z)|\leq \Lambda |z| \;, \qquad \forall z\in B\;, \; \forall n\in \N\; ,
\]
for some $\Lambda<1$. Assume that there exist $r>0$ and a bounded sequence of holomorphic germs
\[
h_n \colon  B_r \rightarrow \C^d\; , \qquad n\in \N\; ,
\]
such that $h_n(0)=0$, $Dh_n(0)$ is unitary, and
\begin{equation}
\label{conjuqua}
h_{n+1} \circ f_n = g_n \circ h_n  \; , \qquad \forall n\in \N\; ,
\end{equation}
as germs at $0$. Then $Wf$ is biholomorphic to $Wg$.
\end{lem}

We conclude this section by stating a result which says that two bounded sequences of germs which are conjugated (in the non-autonomous sense) as jets of a sufficiently high degree are actually conjugated as germs (see \cite[Theorem A.1]{aam11}):

\begin{thm}
\label{volata}
Let $(f_n)$ and $(g_n)$ be two bounded sequences of germs at $0$ of holomorphic self-maps of $\C^d$, such that $f_n(0)=g_n(0)=0$ for every $n\in \N$. Assume that the sequence of linear endomorphisms $(Df_n(0))$
is $(\Lambda,M)$-pinched for some $\Lambda < 1 < M$. Let $k$ be a positive integer such that
\[
\Lambda^{k+1} M < 1\;.
\]
Assume that the $k$-jets of $(f_n)$ and $(g_n)$ are boundedly conjugated, meaning that there exists a bounded sequence of polynomial maps $H_n \colon  \C^d \rightarrow \C^d$, $n\in \N$, of degree at most $k$ with $(DH_n(0)^{-1})$ bounded, such that
\begin{equation}
\label{kconju}
H_{n+1} \circ f_n  =  g_n \circ H_n \quad \mbox{as }k \mbox{-jets, } \qquad \forall n\in \N\;.
\end{equation}
Then $(f_n)$ and $(g_n)$ are boundedly conjugated as germs: There exists a bounded sequence of germs $(h_n)$ such that each germ $h_n$ is invertible, the sequences of inverses $(h_n^{-1})$ is also bounded, the $k$-jet of $h_n$ is $H_n$, and 
\begin{equation}
\label{gconju}
h_{n+1} \circ f_n = g_n \circ h_n\; , \qquad \forall n\in \N\; ,
\end{equation}
as germs. 
\end{thm}

\begin{rem}
As a consequence of (\ref{kconju}), we have the identity
\[
Dg_{n,m}(0) = DH_n(0) \circ D f_{n,m}(0) \circ D H_m(0)^{-1}, \qquad \forall n\geq m \geq 0,
\]
which implies that also the sequence $(Dg_n(0))$ is $(\Lambda,M)$-pinched.
Actually, in \cite[Theorem A.1]{aam11} this theorem is stated and proved under the extra assumptions that $Dg_n(0)=Df_n(0)$ and $DH_n(0)=I$. In this case, the conjugacy $(h_n)$  which one obtains satisfies $Dh_n(0)=I$, and the boundedness of $(h_n^{-1})$ follows from that of $(h_n)$. The more general case stated here follows immediately by applying the more particular statement to the sequences $(f_n)$ and $(DH_{n+1}(0)^{-1} \circ g_n \circ DH_n (0))$, which do have the same linear part and are conjugated as $k$-jets by polynomial mappings tangent to the identity at $0$.
\end{rem} 

\section{Proof of the main theorem}

We are finally ready to prove the main result of this paper:

\begin{thm}
\label{principale}
Let $B$ be the unit ball about the origin in $\C^d$ and let 
\[
f = (f_n : B \rightarrow B)_{n\in \N}
\]
be a sequence of holomorphic maps such that
\begin{equation}
\label{pincio}
M^{-1} |z| \leq |f_n(z)| \leq \Lambda |z|, \qquad \forall z\in B, \; \forall n\in \N,
\end{equation}
where $0<\Lambda<1<M$. Assume that the bunching condition
\begin{equation}
\label{pincio2}
\Lambda^{2+\epsilon(d)} M < 1,
\end{equation}
holds, where $\epsilon(2)=1/14$ and
\begin{equation}
\label{epsi}
\epsilon(d) := \frac{2^{d-1}-1}{2^{2d(d-1)^2} - 2^{d(d-1)^2} - 2^{d-1} + 1}, \qquad \forall d\geq 3. 
\end{equation}
Then the abstract basin of attraction $Wf$ is biholomorphic to $\C^d$.
\end{thm}

\begin{proof}
By the assumption (\ref{pincio}), the differentials of $f_n$ at $0$ satisfy
\[
\|Df_n(0)\|\leq \Lambda, \qquad \|Df_n(0)^{-1}\|\leq M,
\]
and in particular the sequence of linear automorphisms $(Df_n(0))$ is $(\Lambda,M)$-pinched. Let us check that the pair $(\Lambda,M)$ satisfies the assumption (\ref{conve}) of Proposition \ref{con2jets}. A simple algebraic manipulation gives
\[
\begin{split}
(\Lambda^2 M)^{D^{d(d-1)}} (\Lambda M)^{-\delta} &= (\Lambda^2 M)^{D^{d(d-1)}} (\Lambda M)^{- \frac{D-1}{D^{d(d-1)} -1}}\\  &= \left( \Lambda^{2+ \frac{D-1}{D^{2d(d-1)} - D^{d(d-1)} - D + 1}} M \right)^{\frac{D^{2d(d-1)}-D^{d(d-1)}- D + 1}{D^{d(d-1)} - 1}}.
\end{split}
\]
The outer exponent is positive, so the above quantity is less than 1 if and only if the number
\begin{equation}
\label{numb}
\Lambda^{2+ \frac{D-1}{D^{2d(d-1)} - D^{d(d-1)} - D + 1}} M
\end{equation}
is less than 1. The exponent of $\Lambda$ in the above number can be rewritten as
\[
2 + \frac{2^{d-1}-1}{2^{2d(d-1)^2} - 2^{d(d-1)^2} - 2^{d-1} + 1} ,
\]
so it coincides with $2+\epsilon(d)$ when $d\geq 3$, while for $d=2$ it coincides with $2+1/11$, which is larger than $2+\epsilon(2) = 2 + 1/14$. Then (\ref{pincio2}) implies that (\ref{numb}) is less than 1, as we wished to show.

The assumption of Proposition \ref{con2jets} being fulfilled, we obtain sequences of polynomial maps of degree at most 2, $(H_n)$ and $(g_n)$ such that:
\renewcommand{\theenumi}{\roman{enumi}}
\renewcommand{\labelenumi}{(\theenumi)}
\begin{enumerate}
\item $H_n(0)=g_n(0)=0$ and $DH_n(0)$ is unitary;
\item $H_{n+1} \circ f_n = g_n \circ H_n$ as 2-jets, for every $n\in \N$.
\item $U_{\tau_n} \circ g_n \circ U_{\tau_n}^{-1}$ is a special triangular automorphism of $\C^d$ of degree 2. 
\item for every $R > (\Lambda^2 M)^{2^{3d(d-1)^2/2} - 1}$ there is a number $C=C(R)$ such that
\[
\|H_n\|\leq C \, R^n, \qquad \|g_n\| \leq C \, R^n,
\]
for every $n\in \N$.
\end{enumerate}
We claim that we can find a positive number $R$ which satisfies the inequalities
\begin{eqnarray}
\label{ine1} 
R &>& (\Lambda^2 M)^{2^{3d(d-1)^2/2} - 1}, \\
\label{ine2}
R \,\Lambda  &<& 1, \\
\label{ine3}
R^2 \Lambda^3 M &<& 1.
\end{eqnarray}
Indeed, since $\Lambda M\geq 1$, we have
\[
(R \Lambda)^2 \leq R^2 \Lambda^2 ( \Lambda M) =   R^2 \Lambda^3 M,
\]
so (\ref{ine3}) implies (\ref{ine2}). Therefore, it is enough to show that there is some $R$ which satisfies (\ref{ine1}) and (\ref{ine3}), or equivalently that
\begin{equation}
\label{ine}
( \Lambda^2 M)^{2 (2^{3d(d-1)^2/2}-1)} \Lambda^3 M < 1.
\end{equation}
A simple algebraic manipulation shows that the left-hand side of (\ref{ine}) equals
\[
\left( \Lambda^{2+\frac{1}{2(2^{3d(d-1)^2/2}-1)}} M \right)^{2(2^{3d(d-1)^2/2}-1)+1}.
\]
Since the outer exponent is positive, the above quantity is less than 1 because of (\ref{pincio2}) and of the inequality
\[
\frac{1}{2(2^{3d(d-1)^2/2}-1)}\geq \epsilon(d), \qquad \forall d\geq 2,
\]
which is easy to check. Therefore, (\ref{ine}) holds.
 
Let us fix a number $R$ which satisfies (\ref{ine1}), (\ref{ine2}) and (\ref{ine3}), and let $C=C(R)$ be such that the estimates in (iv) hold. For every $n\in \N$ we rescale the maps $f_n$, $g_n$ and $H_n$ as follows:
\begin{eqnarray*}
\tilde{f}_n (z) &:=& R^{n+1} f_n ( R^{-n} z ), \\
\tilde{g}_n (z) &:=& R^{n+1} g_n ( R^{-n} z ), \\
\tilde{H}_n (z) &:=& R^n H_n ( R^{-n} z ).
\end{eqnarray*}
With such definitions, $\tilde{H}_n(0)=\tilde{g}_n(0)=0$, $D\tilde{H}_n(0) = D H_n(0)$ is unitary, and the identity
\begin{equation}
\label{coneq}
\tilde{H}_{n+1} \circ \tilde{f}_n = \tilde{g}_n \circ \tilde{H}_n , \qquad \forall n\in \N,
\end{equation}
holds in the space of 2-jets.

Let us study the properties of the sequence of holomorphic maps $(\tilde{f}_n)$. The differential of $\tilde{f}_n$ at $0$ is
\[
D\tilde{f}_n(0) = R \, Df_n(0),
\]
so it satisfies the estimates
\begin{equation}
\label{tildi}
\|D\tilde{f}_n(0)\| \leq R \,\Lambda =: \tilde{\Lambda} < 1, \qquad
\|D\tilde{f}_n(0)^{-1}\| \leq R^{-1} M =: \tilde{M},
\end{equation}
where we have used (\ref{ine2}).
By (\ref{pincio}), the Cauchy formula implies that $Df_n$ is uniformly bounded on $B_{1/2}$. Therefore
\[
D\tilde{f}_n(z) = R \, Df_n(R^{-n} z)
\]
is uniformly bounded on $B_{1/2}$, hence $(\tilde{f}_n)$ is a bounded sequence of germs. Fix a number $\hat\Lambda$ such that $\tilde{\Lambda}< \hat\Lambda < 1$.
Together with the first of the bounds in (\ref{tildi}), a further use of the Cauchy formula implies that there exists $0<r\leq 1$ such that for every $n\in \N$ 
\begin{equation}
\label{qw1}
|\tilde{f}_n(z)|\leq \hat\Lambda |z|\; ,\quad \forall z\in B_r\;.
\end{equation}
In particular, 
\[
\tilde{f} := (\tilde{f}_n|_{B_r} : B_r \rightarrow B_r)_{n\in \N}
\]
can be seen as an object of $\mathscr{G}$. The maps 
\[
\varphi_n\colon B_r \rightarrow B\;, \quad \varphi_n(z):= R^{-n} z\; , 
\]
define a morphism $\varphi\colon  \tilde{f} \rightarrow f$, which induces a holomorphic injection $W\varphi\colon  W\tilde{f} \rightarrow Wf$. By (\ref{pincio}) and since $\Lambda<R^{-1}$, for every $m\in \N$ there is a natural number $n\geq m$ so large that
\[ 
f_{n,m}(B) \subset B_{\Lambda^{n-m}} \subset B_{r\, R^{-n}} = \varphi_n(B_r)\;.
\]
Hence, Lemma \ref{biholo} implies that $W\varphi$ is a biholomorphism. Therefore, it is enough to show that $W\tilde{f}$ is biholomorphic to $\C^d$. 

On the other hand, $(\tilde{g}_n)$ is a sequence of polynomial maps of degree at most 2. By the conjugacy equation (\ref{coneq}) and since $D\tilde{H}_n(0)$ is unitary, $D\tilde{g}_n(0)$ is related to $D \tilde{f}_n(0)$ by left and right multiplication by unitary automorphisms, so by (\ref{tildi})
\begin{equation}
\label{line}
\|D\tilde{g}_n(0)\| = \|D \tilde{f}_n(0)\| \leq \tilde{\Lambda}, \qquad 
\|D\tilde{g}_n(0)^{-1}\| = \|D \tilde{f}_n(0)^{-1}\| \leq \tilde{M}, \qquad
\forall n\in \N.
\end{equation}
The 2-homogeneous part $\tilde{G}_n$ of $\tilde{g}_n$ is related to the 2-homogeneous part $G_n$ of $g_n$ by
\begin{equation}
\label{tildi2}
\tilde{G}_n (z) = R^{n+1} G_n ( R^{-n} z) = R^{n+1} R^{-2n} G_n(z) = R^{1-n} G_n(z),
\end{equation}
so (iv) implies that $(\tilde{G}_n)$ is bounded. Similarly, (iv) implies that the sequence of polynomial maps (of degree at most 2) $(\tilde{H}_n)$ is bounded.

By the boundedness of $(\tilde{G}_n)$, using also (\ref{line}), and up to the choice of a smaller $r>0$, the Cauchy formula implies that
\begin{equation}
\label{qw2}
|\tilde{g}_n(z)| \leq \hat\Lambda |z|\; , \quad \forall z\in B_r\; .
\end{equation}
In particular, $(\tilde{g}_n)$ defines a bounded sequence in $\mathscr{G}$. By (iii), (\ref{tildi2}), and since the sequence of permutations $(\tau_n)$ satisfies
\[
\tau_{n+1} \neq \tau_n \quad \mbox{if and only if } \quad n=2^{h(d-1)}-1 \mbox{ for some } h\in \N,
\]
the sequence $(\tilde{g}_n)$ satisfies the assumptions of Theorem \ref{treni}, with
\[
k_j = 2^{d-j} \quad \mbox{for }j=1,\dots, d, \qquad K = \max \{k_1,\dots, k_d\} = 2^{d-1}.
\]
Therefore, the basin of attraction of $0$ with respect to $(\tilde{g}_n)$ is the whole $\C^d$:
\begin{equation}
\label{bacino}
\set{ z\in \C^d}{\tilde{g}_{n,0}(z) \rightarrow 0 \mbox{ for } n\rightarrow \infty} = \C^d.
\end{equation}

Since by (\ref{ine3})
\[
\tilde{\Lambda}^{3} \tilde{M} = R^2 \Lambda^3 M <1\; ,
\]
by the bounds (\ref{tildi}), Theorem \ref{volata} implies that $(\tilde{f}_n)$ and $(\tilde{g}_n)$ are boundely conjugated as germs: there exists a bounded sequence $(\tilde{h}_n)\subset \mathscr{G}$ whose sequence of 2-jets is $(\tilde{H}_n)$ such that
\[
\tilde{h}_{n+1} \circ \tilde{f}_n = \tilde{g}_n \circ \tilde{h}_n , \qquad \forall n\in \N,
\]
as germs. Then, by (\ref{qw1}) and (\ref{qw2}), Lemma \ref{conimpW} implies that $W\tilde{f}$ is biholomorphic to $W\tilde{g}$.
Since
\[
W\tilde{g} \cong \set{z\in \C^d}{\tilde{g}_{n,0}(z)\rightarrow 0\mbox{ for } n\rightarrow \infty} = \C^d\; ,
\]
by (\ref{bacino}), we conclude that $W\tilde{f}$ is biholomorphic to $\C^d$ and hence so is $Wf$.
\end{proof}

We can restate the main theorem of the Introduction as a corollary of the above theorem. 

\begin{cor}
Let $f:X\rightarrow X$ be a holomorphic automorphism of a complex manifold and let $K\subset X$ be a compact hyperbolic invariant set with stable distribution $E_s$ of complex dimension $d$. Let $0<\Lambda<1<M$ and $C>0$ be such that
\begin{equation}
\label{lapin}
\max_{x\in K} \bigl\|Df^n(x)|_{E_s}\| \leq C \, \Lambda^n, \qquad 
\max_{x\in K} \bigl\|\bigl(Df^n(x)|_{E_s}\bigr)^{-1} \bigr\| \leq C\, M^n, \qquad \forall n\in \N.
\end{equation}
Assume that the bunching condition
\[
\beta:= \frac{\log M}{-\log \Lambda} < 2 + \epsilon(d)
\] 
holds, where $\epsilon(d)>0$ is defined in Theorem \ref{principale}. Then the stable manifold $W^s(x)$ of every $x\in K$ is biholomorphic to $\C^d$.
\end{cor}

The argument which allows to deduce such a result from the non-autonomous version is standard, but we include it for sake of completeness.

\begin{proof} 
Since $\Lambda^{2+\epsilon(d)} M<1$, we can find positive numbers $1>\tilde{\Lambda}>\Lambda$ and $\tilde{M}>M$ such that 
\begin{equation}
\label{llll}
\tilde{\Lambda}^{2 + \epsilon(d)} \tilde{M} < 1.
\end{equation}
By (\ref{lapin}), up to the replacement of $f$ with a sufficiently hight iterate - an operation which does not change the stable manifolds - we may assume that
\begin{equation}
\label{lapin2}
\max_{x\in K} \bigl\|Df(x)|_{E_s}\| < \tilde{\Lambda}, \qquad 
\max_{x\in K} \bigl\|\bigl(Df(x)|_{E_s}\bigr)^{-1} \bigr\| < \tilde{M}.
\end{equation}
Fix $x\in K$.
By the local stable manifold theorem, we can find a positive number $r$ and
holomorphic embeddings 
\[
\varphi_n : B_r \hookrightarrow W^s ( f^n(x) ) = f^n(W^s(x)),
\]
with domain the ball of radius $r$ about 0 in $\C^d$, mapping 0 into $f^n(x)$, with $D\varphi_n(0)$ an isometry from $\C^d$ to $T_{f^n(x)} W^s(f^n(x))$, and
such that for each $y\in W^s(x)$ the point $f^n(y)$ belongs to $\varphi_n(B_r)$ for $n$ large enough. The identities
$f \circ \varphi_n = \varphi_{n+1} \circ  f_n$ define holomorphic maps $f_n : B_r \rightarrow B_r$ such that $f_n(0)=0$ and, by (\ref{lapin2}), 
\[
\sup_{n\in \N} \bigl\|Df_n(0)\bigr\|< \tilde{\Lambda} , \qquad \sup_{n\in \N} \bigl\|Df_n(0)^{-1}\bigr\| < \tilde{M}.
\]
Up to the choice of a smaller $r$, we may assume that
\[
\sup_{\substack{z\in B_r\\ n\in \N}} \bigl\|Df_n(z)\bigr\|\leq \tilde{\Lambda} , \qquad \sup_{\substack{z\in B_r\\ n\in \N}}  \bigl\|Df_n(z)^{-1}\bigr\| \leq \tilde{M},
\]
from which we get
\begin{equation}
\label{striz}
\tilde{M}^{-1} |z| \leq |f_n(z)| \leq \tilde{\Lambda} |z|, \qquad \forall z\in B_r, \; \forall n\in \N.
\end{equation}
The stable manifold $W^s(x)$ is biholomorphic to the
abstract basin of attraction of $0$ with respect to $(f_n)$ through the map which sends $y\in W^s(x)$ into the equivalence class of the sequence $(\varphi_n^{-1}(f^n(y)))_{n\geq n_0}$, where $n_0\in \N$ is so large that $f^n(y)\in \varphi_n(B_r)$ for every $n\geq n_0$. By (\ref{llll}) and (\ref{striz}), the conclusion follows from Theorem \ref{principale}
\end{proof}

\section{Concluding remarks}
\label{conrem}

The exponents $\epsilon(d)$ become very small already for small values of $d$: For instance, $\epsilon(3) = 1/5591039$. Here we wish to discuss some modifications of our argument which may improve the constants $\epsilon(d)$.

The main part of the proof of Theorem \ref{principale} consists in finding a  conjugacy $(H_n)$ at the level of 2-jets between the original sequence $(f_n)$  and a sequence of 2-jets $(g_n)$ such that
\[
U_{\sigma_h} \circ g_n \circ U_{\sigma_h}^{-1}
\]
is a special triangular automorphism for $m_h \leq n < m_{h+1}$ (in our proof, $m_h = 2^{(d-1)h}-1$ and $(\sigma_h)$ is the sequence of permutations of $\{1,\dots,d\}$ which is introduced in Lemma \ref{combi}). 

A first idea could be to try to optimize the choice of $(m_h)$ and $(\sigma_h)$ - which in our proof are fixed once and for all - to the specific form of the sequence $(f_n)$. In fact, in general the ``best'' pair of sequences $(\sigma_h)$, $(m_h)$, that is the one for which one obtains the existence of a solution $(H_n)$, $(g_n)$ of the 2-jet conjugacy equation with the lowest exponential growth, depends on the sequence of linear maps $(Df_n(0))$. However, it can be shown that the ``worst'' possible sequence $(Df_n(0))$ - the one for which the exponential growth of the least growing solution $(H_n)$, $(g_n)$
is largest among all the $(\Lambda,M)$-pinched sequences of linear automorphisms of $\C^d$ - requires sequences $(\sigma_h)$ and $(m_h)$  which would give the same exponential growth for the least growing solution $(H_n)$, $(g_n)$ associated to any $(\Lambda,M)$-pinched sequence $(Df_n(0))$. Therefore, fixing the sequences $(m_h)$ and $(\sigma_h)$ does not worsen the final result.

A first technical improvement comes from sharpening the estimates of Proposition \ref{con2jets}, which are indeed not optimal: The optimal estimates in this proposition can be derived by generalizing the estimate of Lemma \ref{treniS} (ii) to sequences $S(m,n)$ with more general end-points $m,n$ and plugging the generalized estimate into the proof of Proposition \ref{operatore}. 

A second improvement comes from particularizing Theorem \ref{treni} to the specific situation in which $(T_h)$ is the sequence of permutation automorphisms $(U_{\sigma_h})$. Indeed, since many of the $U_{\sigma_h}$'s permute only a subset of the variables $z_1,\dots,z_d$, one expects the basin of attraction to be the whole $\C^d$ under growth assumptions on $(m_h)$ which are weaker than (\ref{lunghi}). A weaker growth of $(m_h)$ allows the conjugacy equation to be solvable under a milder bunching condition.

Another way of reducing the growth of the sequence of $(m_h)$ in Theorem \ref{treni} is to modify the part of degree higher than 2 of the special triangular automorphisms $f_n$ in order to keep the degree of $f_{m_{h+1},m_h}$ lower than the stable degree $K$. Since in the application of Theorem \ref{volata} we only need a bounded conjugacy at the level of 2-jets, we are actually allowed to work with such modified maps in the proof of Theorem \ref{principale}.

We have carefully checked all these possibilities: They lead indeed to improvements of the sequence of exponents $\epsilon(d)$, which however remains infinitesimal for $d\rightarrow \infty$. On the other hand, such improvements also require considerable complications of the whole argument, which might hide the main ideas, and hence we have decided not to include them.

\providecommand{\bysame}{\leavevmode\hbox to3em{\hrulefill}\thinspace}
\providecommand{\MR}{\relax\ifhmode\unskip\space\fi MR }
\providecommand{\MRhref}[2]{%
  \href{http://www.ams.org/mathscinet-getitem?mr=#1}{#2}
}
\providecommand{\href}[2]{#2}


\begin{thebibliography}{PVW08}

\bibitem[AAM11]{aam11}
A.~Abate, A.~Abbondandolo, and P.~Majer, \emph{Stable manifolds for holomorphic
  automorphisms}, {\tt arXiv:1104.4561v1 [math.DS]}, 2011.
  
\bibitem[Aro11]{aro11}
L.~Arosio, \emph{Basins of attraction in {L}oewner equations}, {\tt
  arXiv:1108.6000v1 [math.CV]}, 2011.  

\bibitem[BDM08]{bdm08}
F.~Berteloot, C.~Dupont, and L.~Molino, \emph{Normalization of bundle
  holomorphic contractions and applications to dynamics}, Ann. Inst. Fourier
  \textbf{58} (2008), 2137--2168.

\bibitem[Bed00]{bed00}
E.~Bedford, \emph{Open problem session of the biholomorphic mappings meeting at
  the american institute of mathematics}, Palo Alto, CA, July 2000.

\bibitem[BW10]{bw10}
K.~Burns and A.~Wilkinson, \emph{On the ergodicity of partially hyperbolic
  systems}, Ann. of Math. (2) \textbf{171} (2010), 451--489.

\bibitem[For99]{for99}
F.~Forstneric, \emph{Interpolation by holomorphic automorphisms and embeddings
  in $\mathbb{C}^n$}, J. of Geom. Anal. \textbf{9} (1999), 93--117.

\bibitem[For04]{for04}
J.~E. Forn{\ae}ss, \emph{Short {$\Bbb C^k$}}, Complex analysis in several
  variables---{M}emorial {C}onference of {K}iyoshi {O}ka's {C}entennial
  {B}irthday, Adv. Stud. Pure Math., vol.~42, Math. Soc. Japan, 2004,
  pp.~95--108.

\bibitem[FS04]{fs04}
J.~E. Forn{\ae}ss and B.~Stens{\o}nes, \emph{Stable manifolds of holomorphic
  hyperbolic maps}, Internat. J. Math. \textbf{15} (2004), 749--758.

\bibitem[HHU07]{rru07}
F.~Rodriguez Hertz, M.~A.~Rodriguez Hertz, and R.~Ures, \emph{A survey of
  partially hyperbolic dynamics}, Partially hyperbolic dynamics, laminations,
  and {T}eichm\"uller flow, Fields Inst. Commun., vol.~51, Amer. Math. Soc.,
  Providence, RI, 2007, pp.~35--87.

\bibitem[HK02]{hk02}
B.~Hasselblatt and A.~Katok (eds.), \emph{Handobbok of dynamical systems},
  vol.~1A, North-Holland, Amsterdam, 2002.

\bibitem[JV02]{jv02}
M.~Jonsson and D.~Varolin, \emph{Stable manifolds of holomorphic
  diffeomorphisms}, Invent. Math. \textbf{149} (2002), 409--430.

\bibitem[KH95]{kh95}
A.~Katok and B.~Hasselblatt, \emph{Introduction to the modern theory of
  dynamical systems}, Cambridge University Press, 1995.

\bibitem[KN11]{kn11}
A.~Katok and V.~Ni\c{t}ic\u{a}, \emph{Rigidity in higher rank {A}belian group
  actions}, vol.~1, Cambridge University Press, 2011.

\bibitem[Pet05]{pet05}
H.~Peters, \emph{Non-autonomous complex dynamical systems}, Ph.D. thesis,
  University of Michigan, 2005.

\bibitem[Pet07]{pet07}
H.~Peters, \emph{Perturbed basins of attraction}, Math. Ann. \textbf{337} (2007),
  1--13.

\bibitem[PVW08]{pvw08}
H.~Peters, L.~R. Vivas, and E.~F. Wold, \emph{Attracting basins of volume
  preserving automorphisms of $\mathbb{C}^k$}, Internat. J. Math. \textbf{19}
  (2008), 801--810.

\bibitem[PW05]{pw05}
H.~Peters and E.~F. Wold, \emph{Non-autonomous basins of attractions and their
  boundaries}, J. Geom. Anal. \textbf{15} (2005), 123--136.

\bibitem[RR88]{rr88}
J.-P. Rosay and W.~Rudin, \emph{Holomorphic maps from {$\mathbb{C}^ n$} to
  {$\mathbb{C}^n$}}, Trans. Amer. Math. Soc. \textbf{310} (1988), 47--86.

\bibitem[Shu87]{shu87}
M.~Shub, \emph{Global stability of dynamical systems}, Springer, New York,
  1987.

\bibitem[SS78]{ss78}
L.~A. Steen and J.~A. Seebach, \emph{Counterexamples in topology},
  Springer-Verlag, New York - Heidelberg, 1978.

\bibitem[Wei97]{wei97}
B.~Weickert, \emph{Automorphisms of $\mathbb{C}^n$}, Ph.D. thesis, University
  of Michigan, 1997.

\bibitem[Wol05]{wol05}
E.~F. Wold, \emph{Fatou-{B}ieberbach domains}, Internat. J. Math. \textbf{16}
  (2005), 1119--1130.

\bibitem[Yoc95]{yoc95b}
J.-C. Yoccoz, \emph{Introduction to hyperbolic dynamics}, Real and complex
  dynamical systems (Hiller{\o}d, 1993), NATO Adv. Sci. Inst. Ser. C Math.
  Phys. Sci., vol. 464, Kluwer Acad. Publ., 1995, pp.~265--291.

\end{thebibliography}
\end{document}